\documentclass{amsart}
\usepackage{amsfonts,amsthm,amsmath}
\usepackage{amssymb}
\usepackage{amscd}
\usepackage[utf8]{inputenc}
\usepackage[a4paper]{geometry}
\usepackage{enumerate}
\usepackage[usenames,dvipsnames]{color}
\usepackage{array}
\usepackage{array,tabularx}

\usepackage{amsthm}

\usepackage{tikz}
\usetikzlibrary{arrows, intersections, calc, matrix}

\DeclareMathOperator{\Coeff}{Coeff}
\DeclareMathOperator*{\Res}{Res}
\newcommand\note[1]{\mbox{}\marginpar{ \scriptsize\raggedright
\hspace{1pt}\color{red} #1}}

\numberwithin{equation}{section}
\numberwithin{equation}{subsection}

\theoremstyle{plain}
\newtheorem*{theorem*}{Theorem}

\newtheorem{theorem}[equation]{Theorem}
\newtheorem{lemma}[equation]{Lemma}
\newtheorem{proposition}[equation]{Proposition}

\newtheorem{corollary}[equation]{Corollary}

\newtheorem{thm}[equation]{Theorem}
\newtheorem{cor}[equation]{Corollary}

\theoremstyle{definition}
\newtheorem{example}[equation]{Example}
\newtheorem{remark}[equation]{Remark}

\newtheorem{definition}[equation]{Definition}

\newtheorem{defn}[equation]{Definition}

\newtheorem{bekezdes}[subsubsection]{}

\DeclareMathOperator{\Supp}{Supp}

\def\C{\mathbb C}
\def\Q{\mathbb Q}
\def\R{\mathbb R}
\def\Z{\mathbb Z}

\def\im{{\rm Im}}

\def\ZZ{\mathbb{Z}}
\def\QQ{\mathbb{Q}}

\newcommand{\cale}{{\mathcal E}}
\newcommand{\calw}{{\mathcal W}}

\newcommand{\calv}{{\mathcal V}}

\newcommand{\calt}{{\mathcal T}}

\newcommand{\cali}{{\mathcal I}}

\newcommand{\calO}{{\mathcal O}}

\newcommand{\calS}{{\mathcal S}}

\newcommand{\calL}{\mathcal{L}}
\newcommand{\caln}{\mathcal{N}}
\newcommand{\s}{r}
\newcommand{\tX}{\widetilde{X}}

\newcommand{\cX}{{\mathcal X}}
\newcommand{\cO}{{\mathcal O}}
\newcommand{\bP}{{\mathbb P}}
\newcommand*{\linebundle}{\mathcal{L}}
\newcommand{\bC}{{\mathbb C}}
\newcommand{\cF}{{\mathcal F}}

\newcommand{\eca}{{\rm ECa}}
\newcommand{\pic}{{\rm Pic}}

\newcommand{\V}{\calv}

\newcommand{\m}{\mathfrak{m}}

\newcommand{\bt}{{\mathbf t}}

\newcommand{\bZ}{{\mathbb{Z}}}
\newcommand{\bQ}{{\mathbb{Q}}}

\author[T. L\'aszl\'o]{Tam\'as L\'aszl\'o}
\address{Babe\c{s}-Bolyai University, Faculty of Mathematics and Computer Science,\newline \hspace*{4mm} 
Str. Mihail Kog\u{a}lniceanu nr. 1, 400084 Cluj-Napoca, Romania}
\email{laszlo.tamas@math.ubbcluj.ro}

\author[J. Nagy]{J\'anos Nagy}
\address{Alfréd R\'enyi Institute of Mathematics,\newline \hspace*{4mm}
MTA-BME Lendulet Arithmetic Combinatorics Research Group}
\email{janomo4@gmail.com}

\title{Brill-Noether problem on splice quotient singularities and duality of topological Poincaré series}

\begin{document}



\begin{abstract}
In this manuscript we investigate the analouge of the Brill-Noether problem for smooth curves in the case of normal surface singularities.
We determine the maximal possible value of $h^1$ of line bundles without fixed components in the Picard group $\pic^{l'}(\tX)$ in the following cases:
for some special Chern classes $l'$ if $\tX$ is a resolution of a splice quotient singularity $(X, 0)$ and for arbitrary  Chern classes in the case of weighted homogenous
singularities.
Motivated by this problem, we define the \emph{virtual cohomology numbers} $h^1_{virt}(l')$ for all Chern classes $l'$ such that $h^1_{virt}(0)$ is the canonical normalized Seiberg-Witten invariant and we generalize the duality formulae of Seiberg-Witten invariants obtained by the authors and A. Némethi in \cite{LNNdual}, for the virtual cohomology numbers.
\end{abstract}

\maketitle

\linespread{1.2}


\pagestyle{myheadings} \markboth{{\normalsize  T. L\'aszl\'o and J. Nagy}} {{\normalsize Brill-Noether problem and duality}}

\thanks{}

\section{Introduction}\label{s:intr}

\subsection{} One of the the major topics of the last decade in the theory of normal surface singularities was to compare the analytic invariants of the singularity with the topological invariants associated with the link of the singularity, or equivalently, with the dual 
resolution graph of a given resolution of the singularity. The main subject of this topic is to provide topological formulae,  or at least topological candidates, for several discrete analytic invariants.  
However, when we fix the topological type and vary the analytic structure, most of the analytic invariants can also change, so one can hope to find purely topological formulae just in the case of special analytical families. 

In the series of articles \cite{NNA1} and \cite{NNA2}, the second author and A. N\'emethi developed the theory of Abel maps for surface singularities which is an analouge of the classical theory of
Abel maps for smooth curves. Although several analytic invariants change when one varies the analytic structure, the authors  proved that if one considers a normal surface singularity $(X,0)$ with rational homology sphere link, one of its good resolution $\tX\to X$ with dual  resolution graph $\mathcal{T}$, then for any first Chern class $l' \in H^2(\tX,\Z)$ the cohomology number $h^1$ of the generic line bundle in $\pic^{l'}(\tX)$ is
independent of the analytic type of $(X,0)$ and it can be given by the topological formula $\chi(l') - \min_{0 \leq l} \chi(l' + l)$.

Since the cohomology numbers are semicontinous, the $h^1$ of the generic line bundle has the smallest possible value in $\pic^{l'}(\tX)$, which means that $r$ is the smallest 
number such that the Brill-Noether strata $ \{\calL \in \pic^{l'}(\tX) \ |\  h^1(\tX, \calL) = r\}$ is nonempty.

In the classical case of a genus $g$ smooth curve $C$, Brill-Noether theory investigates the structure of the Brill-Noether strata $W_d^r = \{\calL \in \pic^{d}(C) \ |\  h^0(C, \calL) \geq r+1\}$ 
among the generic curves or among other special families of algebraic curves. The classical Brill-Noether's theorem \cite{ACGH} states that for a genus $g$ generic curve $C$ and numbers $d \geq 1$ then $W_d^r$ is nonempty if and only if 
$g \geq (r+1)(g-d+r)$.  
This means that the maximal $h^0$ of a line bundle in $\pic^{d}(C)$ is $r+1$, where $r$ is the maximal integer such that $g \geq (r+1)(g-d+r)$.

In the case of surface singularities, the second author proved in \cite{pg} that if we fix a singularity  with rational homology sphere link, a good resolution $\tX$, then for a Chern class $l'$ 
the possible values of $h^1$ of line bundles in $\pic^{l'}(\tX)$ fits in an interval of type $[\chi(-l') - \min_{0 \leq l} \chi(-l' + l), M]$. Therefore,  analogously to the classical Brill-Noether's theorem
it is natural to ask the following: 
\begin{center}
\emph{what is the maximal value $M$ for different analytic types $\tX$ and Chern classes $l'$?}
\end{center}

\subsection{} In the sequel, we present the main ideas and results of the article.

First we note that, in fact for a Chern class $l'$ it is enough to determine the maximal value of $h^1$ of line bundles $\calL \in \pic^{l'}(\tX)$ in the case when $\calL$ has no fixed components (or equivalently, when $\calL$ is in the image of the Abel map). 
Indeed, assume that we know these maximal values and denote them by $M_{f, l'}$. Then, by \cite{NNA1} we know that  
\begin{equation*}
\max_{\calL \in \pic^{l'}(\tX)} h^1(\tX, \calL) = \max_{l \geq 0} \{M_{f, l'-l} + \chi(-l') - \chi(-l' + l)\}.
\end{equation*}

In this article, we investigate these numbers $M_{f, l'}$ in the case of weighted homogenous and splice quotient singularities, and we determine them in some special cases. In fact, we have the following results.

\begin{theorem*}[A]
Let $(X, 0)$ be a weighted homogeneous singularity with resolution $\tX$ and consider a cycle in the Lipman cone $l' \in \calS'$.  Then for every line bundle $\calL \in \pic^{-l'}(\tX)$ one has $$h^1(\tX, \calL) \leq h^1(\calO_{\tX}(-l')).$$
\end{theorem*}
In other words, for weighted homogenous singularities the natural line bundle $\calO_{\tX}(-l')$ has largest $h^1$ value among all the line bundles in the Picard group $\pic^{l'}(\tX)$. It turns out that in the case of splice quotient singularities this is not true anymore for all the Chern classes as it is shown by a counterexample in section \ref{ss:cex}. Nevertheless, we can handle the following situation.

\begin{theorem*}[B]
Let $(X,0)$ be a splice quotient singularity with rational homology sphere link and let $\calt$ be its dual resolution graph which satisfies the monomial conditions and $\tX$ the corresponding resolution space with exceptional divisor $E=\{E_v\}_{v\in \calv}$. 

Let $l' = a_v E_v^*$ be a cycle for some vertex $v$ of $\calt$ such that $a_v > 0$ and $E_v^* \in S_{an}$.  \\
(a) \ If $a_v = 1$ one has $h^1(\tX, \calL) \leq h^1( \calO_{\tX}(-E_v^*))$ for every line bundle $\calL \in \im(c^{-E_v^*}(Z))$. \\
(b) \ If $a_v >1$ we also assume that the line bundle $\calO_{\tX}(- E_v^*)$ has no base point on the exceptional divisor $E_v$. In this case there are only finitely many points $p_1, \cdots, p_k \in E_v$ such that the line bundle $\calO_{\tX}(- E_v^*)$ has no section through $p_i$. Then, for any line bundle $\calL = \calO_{\tX}( \sum_{1 \leq k \leq a_v} D_k)$ given by disjoint transversal cuts $D_{ k},  1 \leq k \leq a_v$ at regular points of the exceptional
divisor $E_v$, such that none of the transversal cuts go through the critical points $ p_1, \cdots, p_k$, one has the following inequality:
$$h^1(\tX, \calL) \leq h^1(\tX, \calO(-l')).$$
\end{theorem*}

In order to prove Theorem (B) one has to investigate the arithmetics of the cohomology numbers $h^1( \calO_{\tX}(-l'))$ of natural line bundles on splice quotient singularities.  Our starting point is a theorem of N\'emethi \cite{NCL} which says that in the case of splice quotient singularities the cohomology numbers $h^1( \calO_{\tX}(-l'))$ can be expressed by certain counting functions of the topological Poincaré series $Z(\mathbf{t})$ associated with the resolution graph $\calt$.

Notice that in the very particular case, when $l' = 0$, the cohomology number $h^1(\calO_{\tX})=p_g(\tX)$ is the geometric genus of $(X,0)$ which, by Laufer duality, can be expressed as the dimension of a space of differential forms $\dim\left( H^0(\tX\setminus E, \Omega^2_{\tX})/ H^0(\tX,\Omega_{\tX}^2)  \right)$. This equality motivates and provides a `dual' expression for  $p_g(\tX)$ using the $[Z_K]$-equivariant part of the topological Poincar\'e series $Z_{[Z_K]}(\bt)=\sum_{l'}z(l')\bt^{l'}$ as $p_g(\tX) = \sum_{[l'] = [Z_K], l' \ngeq Z_K} z(l')$, where $Z_K$ is the anti-canonical cycle (cf. section \ref{ss:analinv}).

From topological point of view, the authors in \cite{LNNdual} extended this idea to express for all  rational homology sphere plumbed 3-manifolds $M$ associated with a negative definite plumbing graph $\calt$ the normalized Seiberg-Witten invariants $\mathfrak{sw}_h^{norm}(M)$ of  $M$ as certain values of the counting functions, namely $\mathfrak{sw}_h^{norm}(M) = p_g(\tX) = \sum_{[l'] = [Z_K - r_h], l' \ngeq Z_K - r_h} z(l')$.

In this article, we aim to extend these ideas to two different directions. 

First, in the case of splice quotient singularities (with certain extra properties on their analytic semigroup) we give the connection of cohomology numbers $h^1( \calO_{\tX}(-l'))$ 
and certain spaces of differential forms, and using this approach we deduce the following result:

 \begin{theorem*}[C]
Let $\tX$ be a resolution space of a splice quotient singularity $(X, 0)$ so that its dual resolution graph $\calt$ satisfies the monomial conditions.  We consider a very large cycle $Z$ on it, a cycle  $l' = \sum_{v \in \calv} a_v E_v^* \in L'$ and a divisor $p \in (c^{l'}(Z))^{-1}(0) \subset \eca^{-l'}(Z)$. Thus $\calO_{\tX}(p) = \calO_{\tX}(-l')$ and  assume further that $p$ consists of disjoint transversal cuts. Let $S=\{p_{v, k}\}$ be the set of  intersection points of $p$ with the exceptional divisor $E \subset \tX$, and  for every point $x \in S$ we index by $v_x$ the corresponding component of $E$. 
Let's denote for an arbitrary subset $I \subset S$, $l'_{I} = \sum_{x \in I} E_{v_x}^*$, and let $|l'|^* = J \subset \calv$.

Furthermore we consider the following multivariable series $R(\mathbf{t}) = \prod_{ u\in \calv}(1-\mathbf{t}^{E_u^*})^{\delta_u -2 + a_u} = \sum_{l''} r(l'')  \bt^{l''}$. Then we claim the following identity:

\begin{equation}\label{duall}
\sum_{I \subset S} (-1)^{|I|} h^1(\calO_{\tX}( - l'_I)) = \sum_{l'' \in S',[ l''] = [Z_K], l''_J \leq (Z_K-E)_J} r(l'').
\end{equation}

\end{theorem*} 

Secondly, we define the topological candidate of $h^1( \calO_{\tX}(-l'))$ which is called the \emph{virtual cohomology number} $h^1_{virt}(l')$, and in the completely topological setting we extend the `duality' formula (\ref{duall}) to the general case, see section \ref{s:topdcf} and Theorem \ref{topdualcount}.

%

\subsection{} The paper is structured as follows. 

In section \ref{s:prel} we recall the necessary background material regarding topological and analytic invariants of normal surface singularities, effective Cartier divisors and Abel maps, the theory of  topological Poincaré series and useful surgery properties, and splice quotient singularities.

Section \ref{s:motspq} serves as a motivation of this work, containing an interpretation of the cohomology numbers $h^1( \calO_{\tX}(-l'))$ in terms of diferential forms in the case of splice quotient singularities and we deduce Theorem (C). 

In section \ref{s:topdcf}, we introduce the virtual cohomology number and prove the topological generalization of Theorem (C) which holds for all rational homology sphere link of singulatities, even if they graphs do not satisfy the monomial conditions, or in other words they do not support splice quotient singularities. 

In section \ref{s:domspq} we consider the case of splice quotient singularities and prove Theorem (B) using Theorem (C) about the expression of the cohomology numbers $h^1( \calO_{\tX}(-l'))$ by counting functions. Furthermore, we construct a counterexample showing that Theorem (B) does not hold for all Chern classes unlike  weighted homogenous case. 
In other words we show that there are Chern classes $l' \in S'$ such that there is a line bundle $\calL \in \pic^{-l'}(\tX)$ such that $h^1(\tX, \calL) > h^1(\calO_{\tX}(-l'))$.

Section \ref{s:domwh} discusses the case of weighted homogenous singularities and we prove Theorem (A).

For the last two sections we turn back to the topology and concentrate on some properties of the topological Poincar\'e series. Namely, in section \ref{s:wildprop}, by constructing further examples we emphasize the failure of some important properties of the topological Poincar\'e series when the underlying graph does not satisfy the monomial conditions. An important part of this discussion is an example of a resolution graph 
for which the normalised Seiberg-Witten invariant  is negative (unlike the case of splice quotient singularities when it is certainly nonnegative). 

Finally, in section \ref{s:mctPs} we explore the ideas of the previous section and show how the monomial conditions plays an important role in certain arithmetic properties of the topological Poincaré series. In particular, 
we end our article with the following result:

\begin{theorem*}[D]
Let $\mathcal{T}$ be a resolution graph of a normal surface singularity with a rational homology sphere link, consider a vertex $v$ of $\calt$ and suppose that the monomial conditions holds for branches of nodes which
does not contain the vertex $v$.  Then the canonical normalized Seiberg-Witten invariant is nonnegative.
\end{theorem*}

\section{Preliminaries}\label{s:prel}

\subsection{The resolution}\label{ss:notation}
Let $(X,0)$ be the germ of a complex analytic normal surface singularity,
 and we fix  a good resolution  $\phi:\widetilde{X}\to X$ of $(X,0)$.  
We denote the exceptional curve $\phi^{-1}(0)$ by $E$, and let $\cup_{v\in\calv}E_v$ be
its irreducible components. Set also $E_I:=\sum_{v\in I}E_v$ for any subset $I\subset \calv$. 

\subsection{Topological invariants}\label{ss:topol}
Let $\calt$ be the dual resolution graph
associated with $\phi$;  it  is a connected graph.
Then $M:=\partial \widetilde{X}$ can be identified with the link of $(X,o)$, it is 
an oriented  plumbed 3--manifold associated with $\calt$.
We will assume that  \emph{$M$ is a rational homology sphere},
or, equivalently,  $\mathcal{T}$ is a tree and all genus
decorations of $\mathcal{T}$ are zero. We use the same
notation $\mathcal{V}$ for the set of vertices of $\calt$, and let $\delta_v$ be the valency of a vertex $v$.

$L:=H_2(\widetilde{X},\mathbb{Z})$, endowed
with the negative definite intersection form  $I=(\,,\,)$, is a lattice. It is
freely generated by the classes of 2--spheres $\{E_v\}_{v\in\mathcal{V}}$. The elements $l=\sum n_vE_v \in L$ are called \emph{cycles} and we define their $E$-support by $|l|=\cup_{n_v\not=0}E_v$. 
 Then $L':= H^2(\widetilde{X},\mathbb{Z})$ is generated
by the (anti)dual classes $\{E^*_v\}_{v\in\mathcal{V}}$ defined
by $(E^{*}_{v},E_{w})=-\delta_{vw}$, the opposite of the Kronecker symbol.
The intersection form embeds $L$ into $L'$. Then $H_1(M,\mathbb{Z})\simeq L'/L$, abridged by $H$.
Usually one also identifies $L'$ with those rational cycles $l'\in L\otimes \Q$ for which
$(l',L)\in\Z$, or, $L'={\rm Hom}_\Z(L,\Z)$. 

For $l'_1,l'_2\in L\otimes \Q$ with $l'_i=\sum_v l'_{iv}E_v$ ($i=\{1,2\}$)
one considers a partial ordering $l'_1\geq l'_2$ defined coordinatewise by $l'_{1v}\geq l'_{2v}$
for all $v\in\calv$. In particular,
$l'$ is an effective rational cycle if $l'\geq 0$.

Each class $h\in H=L'/L$ has a unique representative $r_h=\sum_vr_vE_v\in L'$ in the semi-open cube
(i.e. each $r_v\in \bQ\cap [0,1)$), such that its class  $[r_h]$ is $h$.

All the $E_v$--coordinates of any $E^*_u$ are strict positive.
We define the Lipman cone as $\calS':=\{l'\in L'\,:\, (l', E_v)\leq 0 \ \mbox{for all $v$}\}$.
It is generated over $\bZ_{\geq 0}$ by $\{E^*_v\}_v$. We will also introduce the notation $\calS:=\calS'\cap L$.

For more details regarding the above combinatorial package associated with the topology of normal surface singularities we refer to \cite{trieste,NCL,Nfive}.

\subsection{Analytic invariants}\label{ss:analinv}
\subsubsection{} The group ${\rm Pic}(\widetilde{X})$
of  isomorphism classes of analytic line bundles on $\widetilde{X}$ appears in the (exponential) exact sequence
\begin{equation}\label{eq:PIC}
0\to {\rm Pic}^0(\widetilde{X})\to {\rm Pic}(\widetilde{X})\stackrel{c_1}
{\longrightarrow} L'\to 0, \end{equation}
where  $c_1$ denotes the first Chern class. Here
$ {\rm Pic}^0(\widetilde{X})=H^1(\widetilde{X},\calO_{\widetilde{X}})\simeq
\C^{p_g}$, where $p_g$ is the {\it geometric genus} of
$(X,0)$. $(X,0)$ is called {\it rational} if $p_g(X,0)=0$.
The works of Artin \cite{Artin62,Artin66} characterized rational singularities topologically
via the graphs; such graphs are called `rational'. By this criterion, $\calt$
is rational if and only if $\chi(l)\geq 1$ for any effective non--zero cycle $l\in L_{>0}$.
Here $\chi(l)=-(l,l-Z_K)/2$ is the Riemann-Roch function and $Z_K\in L'$ is the (anti)canonical cycle
identified by adjunction formulae
$(-Z_K+E_v,E_v)+2=0$ for all $v$.

The epimorphism
$c_1$ admits a unique group homomorphism section $l'\mapsto s(l')\in {\rm Pic}(\widetilde{X})$,
 which extends the natural
section $l\mapsto \calO_{\widetilde{X}}(l)$ valid for integral cycles $l\in L$, and
such that $c_1(s(l'))=l'$  \cite{trieste,OkumaRat}.
We call $s(l')$ the  {\it natural line bundles} on $\widetilde{X}$ and they will be denoted by $\calO_{\tX}(l')$. 
By  the very  definition, $\calL$ is natural if and only if some power $\calL^{\otimes n}$
of it has the form $\calO_{\tX}(l)$ for some $l\in L$.

\subsubsection{$\mathbf{{Pic}(Z)}$}
Similarly, if $Z\in L_{>0}$ is a non--zero effective integral cycle such that its support is $|Z| =E$,
and $\calO_Z^*$ denotes
the sheaf of units of $\calO_Z$, then ${\rm Pic}(Z)=H^1(Z,\calO_Z^*)$ is  the group of isomorphism classes
of invertible sheaves on $Z$. It appears in the exact sequence
  \begin{equation}\label{eq:PICZ}
0\to {\rm Pic}^0(Z)\to {\rm Pic}(Z)\stackrel{c_1}
{\longrightarrow} L'\to 0, \end{equation}
where ${\rm Pic}^0(Z)=H^1(Z,\calO_Z)$.
If $Z_2\geq Z_1$ then there are natural restriction maps,
${\rm Pic}(\widetilde{X})\to {\rm Pic}(Z_2)\to {\rm Pic}(Z_1)$.
Similar restrictions are defined at  ${\rm Pic}^0$ level too.
These restrictions are homomorphisms of the exact sequences  (\ref{eq:PIC}) and (\ref{eq:PICZ}).

Furthermore, we define a section of (\ref{eq:PICZ}) by
$s_Z(l'):=
{\mathcal O}_{\widetilde{X}}(l')|_{Z}$.
It also satisfies $c_1\circ s_Z={\rm id}_{L'}$. We write  ${\mathcal O}_{Z}(l')$ for $s_Z(l')$, and they are called 
 {\it natural line bundles } on $Z$.

We also use the notations ${\rm Pic}^{l'}(\widetilde{X}):=c_1^{-1}(l')
\subset {\rm Pic}(\widetilde{X})$ and
${\rm Pic}^{l'}(Z):=c_1^{-1}(l')\subset{\rm Pic}(Z)$
respectively. Multiplication by $\calO_{\widetilde{X}}(-l')$, or by
$\calO_Z(-l')$, provides natural affine--space isomorphisms
${\rm Pic}^{l'}(\widetilde{X})\to {\rm Pic}^0(\widetilde{X})$ and
${\rm Pic}^{l'}(Z)\to {\rm Pic}^0(Z)$.

\subsubsection{\bf The analytic semigroup} \label{bek:ansemgr}
By definition, the analytic semigroup (monoid)  associated with the resolution $\tX$ is
\begin{equation}\label{eq:ansemgr}
\calS'_{an}:= \{l'\in L' \,:\,\calO_{\tX}(-l')\ \mbox{has no  fixed components}\}.
\end{equation}
It is a subsemigroup of $\calS'$. One also sets $\calS_{an}:=\calS_{an}'\cap L$, a subsemigroup
of $\calS$. In fact, $\calS_{an}$
consists of the restrictions   ${\rm div}_E(f)$ of the divisors
${\rm div}(f\circ \phi)$ to $E$, where $f$ runs over $\calO_{X,0}$. Therefore, if $s_1, s_2\in \calS_{an}$, then
${\rm min}\{s_1,s_2\}\in \calS_{an}$ as well (take the generic linear combination of the corresponding functions).
In particular,  for any $l\in L$, there exists a {\it unique} minimal
$s\in \calS_{an}$ with $s\geq l$.

Similarly, for any $h\in H=L'/L$ set $\calS'_{an,h}:\{l'\in \calS_{an}\,:\, [l']=h\}$.
Then for any  $s'_1, s'_2\in \calS_{an,h}$ one has
${\rm min}\{s'_1,s'_2\}\in \calS_{an,h}$, and
for any $l'\in L'$   there exists a unique minimal
$s'\in \calS_{an,[l']}$ with $s'\geq l'$.

\subsubsection{\bf Special cycles}  We will write $Z_{min}\in L$ for the  {\it minimal} (or fundamental) cycle of Artin, which is
the minimal non--zero cycle of $\calS=\calS'\cap L$ \cite{Artin62,Artin66}. Yau's {\it maximal ideal cycle}
$Z_{max}\in L$ defines the  divisorial part of the pullback of the maximal ideal $\m_{X,o}\subset \calO_{X,o}$, i.e.
 $\phi^*{\m_{X,o}}\cdot \calO_{\widetilde{X}}=\calO_{\widetilde{X}}(-Z_{max})\cdot \cali$,
where $\cali$ is an ideal sheaf with 0--dimensional support \cite{Yau1}. In general $Z_{min}\leq Z_{max}$.

\subsection{Effective Cartier divisors and Abel maps}

  In this section we review some needed material from \cite{NNA1}.

We fix a good resolution $\phi:\tX\to X$ of a normal surface singularity,
whose link is a rational homology sphere. 

\subsubsection{} \label{ss:4.1}
Let us fix an effective integral cycle  $Z\in L$, $Z\geq E$. (The restriction $Z\geq E$ is imposed by the
easement of the presentation, everything can be adopted  for $Z>0$).

Let $\eca(Z)$  be the space of effective Cartier (zero dimensional) divisors supported on  $Z$.
Taking the class of a Cartier divisor provides  a map
$c:\eca(Z)\to \pic(Z)$.
Let  $\eca^{l'}(Z)$ be the set of effective Cartier divisors with
Chern class $l'\in L'$, that is,
$\eca^{l'}(Z):=c^{-1}(\pic^{l'}(Z))$.

We consider the restriction of $c$, $c^{l'}:\eca^{l'}(Z)
\to \pic^{l'}(Z)$ too, sometimes still denoted by $c$. 

For any $Z_2\geq Z_1>0$ one has the natural  commutative diagram
\begin{equation}\label{eq:diagr}
\begin{picture}(200,45)(0,0)
\put(50,37){\makebox(0,0)[l]{$
\eca^{l'}(Z_2)\,\longrightarrow \, \pic^{l'}(Z_2)$}}
\put(50,8){\makebox(0,0)[l]{$
\eca^{l'}(Z_1)\,\longrightarrow \, \pic^{l'}(Z_1)$}}
\put(70,22){\makebox(0,0){$\downarrow$}}
\put(135,22){\makebox(0,0){$\downarrow$}}
\end{picture}
\end{equation}

As usual, we say that $\calL\in \pic^{l'}(Z)$ has no fixed components if
\begin{equation}\label{eq:H_0}
H^0(Z,\calL)_{reg}:=H^0(Z,\calL)\setminus \bigcup_v H^0(Z-E_v, \calL(-E_v))
\end{equation}
is non--empty. 
Note that $H^0(Z,\calL)$ is a module over the algebra
$H^0(\calO_Z)$, hence one has a natural action of $H^0(\calO_Z^*)$ on
$H^0(Z, \calL)_{reg}$. This second action is algebraic and free.  Furthermore,
 $\calL\in \pic^{l'}(Z)$ is in the image of $c$ if and only if
$H^0(Z,\calL)_{reg}\not=\emptyset$. In this case, $c^{-1}(\calL)=H^0(Z,\calL)_{reg}/H^0(\calO_Z^*)$.

One verifies that $\eca^{l'}(Z)\not=\emptyset$ if and only if $-l'\in \calS'\setminus \{0\}$. Therefore, it is convenient to modify the definition of $\eca$ in the case $l'=0$: we (re)define $\eca^0(Z)=\{\emptyset\}$,
as the one--element set consisting of the `empty divisor'. We also take $c^0(\emptyset):=\calO_Z$, then we have
\begin{equation}\label{eq:empty}
\eca^{l'}(Z)\not =\emptyset \ \ \Leftrightarrow \ \ l'\in -\calS'.
\end{equation}
If $l'\in -\calS'$  then
  $\eca^{l'}(Z)$ is a smooth variety whose dimension equals with the intersection number $(l',Z)$. Moreover,
if $\calL\in \im (c^{l'}(Z))$ (the image of the map $c^{l'}$)
then  the fiber $c^{-1}(\calL)$
 is a smooth, irreducible quasiprojective variety of  dimension
 \begin{equation}\label{eq:dimfiber}
\dim(c^{-1}(\calL))= h^0(Z,\calL)-h^0(\calO_Z)=
 (l',Z)+h^1(Z,\calL)-h^1(\calO_Z).
 \end{equation}

\bekezdes \label{bek:I}
Consider again  a Chern class (or cycle) $l'\in-\calS'$ as above.
The $E^*$--support $|l'|^* = J(l')\subset \calv$ of $l'$ is defined via the identity  $l'=\sum_{v\in J(l')}a_vE^*_v$ with all
$\{a_v\}_{v\in J}$ nonzero. Its role is the following.

Besides the Abel map $c^{l'}(Z)$ one can consider its `multiples' $\{c^{nl'}(Z)\}_{n\geq 1}$ as well. It turns out
(cf. \cite[\S 6]{NNA1}) that $n\mapsto \dim \im (c^{nl'}(Z))$
is a non-decreasing sequence, and  $\im (c^{nl'}(Z))$ is an affine subspace for $n\gg 1$, whose dimension $e_Z(l')$ is independent of $n$, and essentially it depends only
on $J(l')$.
We denote the linearization of this affine subspace by $V_Z(J) \subset H^1(\calO_Z)$ or if the cycle $Z \gg 0$, then $ V_{\tX}(J) \subset H^1(\calO_{\tX})$.

Moreover, by \cite[Theorem 6.1.9]{NNA1},
\begin{equation*}\label{eq:ezl}
e_Z(l')=h^1(\calO_Z)-h^1(\calO_{Z|_{\calv\setminus J(l')}}),
\end{equation*}
where $Z|_{\calv\setminus J(l')}$ is the restriction of the cycle $Z$ to its $\{E_v\}_{v\in \calv\setminus J(l')}$
coordinates.

If $Z\gg 0$ (i.e. all its $E_v$--coordinates are very large), then (\ref{eq:ezl}) reads as
\begin{equation*}\label{eq:ezlb}
e_Z(l')=h^1(\calO_{\tX})-h^1(\calO_{\tX(\calv\setminus J(l'))}),
\end{equation*}
where $\tX(\calv\setminus J(l'))$ is a convenient small tubular neighbourhood of $\cup_{v\in \calv\setminus J(l')}E_v$.

Let $\Omega _{\tX}(J)$ be the subspace of $H^0(\tX\setminus E, \Omega^2_{\tX})/ H^0(\tX,\Omega_{\tX}^2)$ generated by differential forms which have no poles along $E_J\setminus \cup_{v\not\in J}E_v$.
Then, cf. \cite[\S8]{NNA1},
\begin{equation*}\label{eq:ezlc}
h^1(\calO_{\tX(\calv\setminus J)})=\dim \Omega_{\tX}(J).
\end{equation*}

Similarly let $\Omega _{Z}()$ be the subspace of $H^0(\calO_{\tX}(K + Z))/ H^0(\calO_{\tX}(K))$ generated by differential forms which have no poles along $E_J\setminus \cup_{v\not\in J}E_v$.
Then, cf. \cite[\S8]{NNA1},
\begin{equation*}\label{eq:ezlc}
h^1(\calO_{Z_{(\calv\setminus J)}})=\dim \Omega_{Z}(J).
\end{equation*}

We have also the following duality from \cite{NNA1} supporting the equalities above:

\begin{theorem}\cite{NNA1}\label{th:DUALVO}
Via Laufer duality one has  $V_{\tX}(J)^*=\Omega_{\tX}(J)$ and $V_{Z}(J)^*=\Omega_Z(J)$.
\end{theorem}

\subsubsection{\bf $\dim \im (c^{l'}(Z))$ and $h^1$ of cycles}

We recall some theorems about the dimensions of images of Abel maps from \cite{NNAD}.

If $\tX$ is a resolution of a normal surface singularity $(X, 0)$, $Z$ is an effective integral cycle on $\tX$ and $l' \in -S'$,  we set the notation $d_{l', Z}:= \dim(\im(c^{l'}(Z)))$.

Then, following \cite{NNA1}, we recall the  interpretation of $d_{l', Z}$ using cohomology numbers.

\begin{theorem}\label{dimgen}\cite{NNA1}
Let us consider the same setup as before: a resolution $\tX$ of a singularity $(X, 0)$ with resolution graph $\mathcal{T}$, $l' \in -S'$ and $Z$ an arbitrary effective cycle. 
Then for an abitrary line bundle $\calL \in \im(c^{l'}(Z))$ we have the inequality $h^1(Z, \calL) \geq  h^1(\calO_Z) - d_{l', Z}$, which becomes an equality for the generic line bundles in $\im(c^{l'}(Z))$.
\end{theorem}

Furthermore, we recall the following formulae for $d_{l', Z}$ from \cite {NNAD}.

\begin{theorem}\label{dimcomp}\cite{NNAD}
\begin{equation}
d_{l', Z} =  \min_{0 \leq Z_1 \leq Z}( (l', Z_1) + h^1(\calO_Z) - h^1(\calO_{Z_1})).
\end{equation}
\end{theorem}
Note also the following  interesting geometric interpretation of the formulae given above

\begin{cor}\cite{NNAD}
Given a resolution $\tX$ with resolution graph $\mathcal{T}$, $l' \in -S'$ and $Z$ an arbitrary effective cycle, there exists a minimal effective cycle $Z_1 \leq Z$ and a maximal effective cycle $Z_2 \leq Z$, such that the map $\eca^{l'}(Z_i) \to \pic^{l'}(Z_i) $ is birational and the generic fibres of the map $\im(c^{l'}(Z)) \to \im(c^{l'}(Z_i))$ have dimensions
$h^1(\calO_Z) - h^1(\calO_{Z_i})$, which is the largest possible. For a fixed choice of $Z, l'$, let us denote $Z_1$ by $C_{min}(Z, l')$ and $Z_2$ by $C_{max}(Z, l')$.
\end{cor}

Finally, as a closure of this section, let recall also the following propositions from \cite{NNAD} and \cite{NNA1} which will be useful later in our proofs.

\begin{proposition}\cite{NNAD}\label{0dim}
Consider a resolution $\tX$ of a singularity $(X, 0)$ with resolution graph $\mathcal{T}$, $l' \in -S'$ and $Z$ an arbitrary effective cycle. 
Assume that  $|l'|^* = J$ and no differential form in $H^0(\tX\setminus E, \Omega^2_{\tX})/ H^0(\tX,\Omega_{\tX}^2)$ has  a pole on an exceptional divisor $E_v, v \in J$, then 
$d_{l', Z} = 0$ and for the unique line bundle $\calL \in \im(c^{l'}(Z))$ one has $h^1(Z, \calL) = p_g$.
\end{proposition}

\begin{proposition}\cite{NNA1}\label{redcycle}
Let  $\tX$ be a resolution of a singularity $(X, 0)$ with resolution graph $\mathcal{T}$ and consider an $l' \in -S'$ and $Z_1 \leq Z_2$ effective cycles  such that $h^1(\calO_{Z_1}) = h^1(\calO_{Z_2})$.
Then we have $d_{l', Z_1} = d_{l', Z_2}$, and if $\calL \in \im(c^{l'}(Z_2))$ then $h^1(Z_2, \calL) = h^1(Z_1, \calL | Z_1)$.
\end{proposition}

\subsection{The periodic constant of multivariable formal Laurent series}
\label{ss:set}
One of the main tool of this article is the theory and methods regarding topological Poincar\'e series (cf. section \ref{ss:tps}), extended and applied to different `relative' multivariable series constructed from the original topological Poincar\'e series. Therefore, in this subsection we will define the important concepts in a slightly more general setting, and in the next section we will specialize them and discuss results regarding the topological Poincar\'e series. 

\subsubsection{} Let $L$ be a lattice freely generated by base elements $\{E_v\}_{v\in \calv}$,
$L' $ is an overlattice of the same rank (not necessarily dual of $L$), and we set $H:=L'/L$, a finite abelian group of order $d$. The partial ordering is defined as in subsection~\ref{ss:topol}.
Let $\ZZ[[L']]$ be the $\ZZ$-module
consisting of the $\ZZ$-linear combinations of the monomials $\mathbf{t}^{l'}:=\prod_{v\in \V}t_v^{l'_v}$,
where $l'=\sum_{v}l'_v E_v\in L'$. It is a $\ZZ$-submodule of the formal power series in variables $t_v ^{1/d}, t_v^{-1/d}$, $v\in V$.

We consider a multivariable series $S(\bt)=\sum_{l'\in L'}a(l')\bt^{l'}\in \ZZ[[L']]$ and
let $\Supp S(\bt):=\{l'\in L' \mid a(l')\neq 0\}$ be the support of the
series and we assume the following finiteness condition: for any $x\in L'$
\begin{equation}\label{eq:finiteness}
\{l'\in \Supp S(\bt)\mid l'\not\geq x\} \ \ \mbox{is finite}.
\end{equation}
Note that this condition implies that $S(\bt)$ is automatically a formal Laurent series in the sense that the subset $\{l'\in L'_{\ngeq 0} \mid a(l')\neq 0\}$ of the support is finite.  

We will use multivariable series in $\ZZ[[L']]$ as well as in $\ZZ[[L'_I]]$ for any
$I\subset \calv$, where $L'_I={\rm pr}_I(L')$ is the projection of $L'$ via
${\rm pr}_I:L_{{\QQ}} \to \oplus_{v\in I} \QQ \langle E_v\rangle$. 
For example, if $S(\bt)\in \ZZ[[L']]$ then $S(\bt_{I}):=S(\bt)|_{t_v=1,v\notin I}$ is an element of $\ZZ[[L'_I]]$.
In the sequel we use the notations $\l'_I=l'|_I:= {\rm pr}_I(l')$ and $\bt^{l'}_I:=\bt^{l'}|_{t_v=1,v\notin I}$ for any $l'\in L'$.
Each coefficient $a_I(x)$ of $S(\bt_I)$ is obtained as a summation of certain coefficients $a(y)$ of $S(\bt)$, where $y$ runs over
$ \{\ell'\in \Supp S(\bt)\mid \ell'|_I =x \}$
(this is a finite sum by~\eqref{eq:finiteness}). Moreover, $S(\bt_I)$ satisfies a similar finiteness property as ~\eqref{eq:finiteness}
in the variables~$\bt_I$.

For any $S(\bt)\in \ZZ[[L']]$ one can consider its unique decomposition  $S(\bt)=\sum_h S_h(\bt)$,
where $S_h(\bt):=\sum_{[l']=h}a(l')\bt^{l'}$. $S_h(\bt)$ is called the
$h$-part of $S(\bt)$.  Note that
 the restriction
$S_h(\bt)|_{t_v=1,v\notin I}$ of the $h$-part $S_h(\bt)$ cannot be recovered from $S(\bt_I)$ in general. 
%

\subsubsection{\bf Counting function}
\label{ss:countingfunctions}
Given a multivariable series $A(\bt_I)\in \ZZ[[L'_I]]$ for $\emptyset\neq I\subset \V$
(eg., $A(\bt_I)=S(\bt_I)$ or $A(\bt_I)=S_h(\bt_I)$ for $h\in H$) one considers the \textit{counting function} associated with the coefficients of $A(\bt_{I})$ as follows:
(cf. \cite{NPS}).
\begin{equation}\label{eq:count1}
Q{(A(\bt_I))}: L'_I\longrightarrow \ZZ, \ \ \ \
x_I\mapsto \sum_{l'_I\ngeq x_I\atop l'_I\nless 0} a(l'_I).
\end{equation}
%

Note that  $Q{(A(\bt_I))}$ is well defined whenever $A$ satisfies the finiteness condition~\eqref{eq:finiteness}.
If $A(\bt_I)=S_h(\bt_I)$ then
$Q{(A(\bt_I))}$ will also be denoted by~$Q^{S}_{h,I}$.

\subsubsection{\bf Periodic constants}\label{ss:pc}
Let $S(\bt) \in \Z[[L']]$ be a series satisfying the finiteness condition~\eqref{eq:finiteness}
and consider its $h$-part $S_h(\bt)$ for a fixed $h \in H$.
Let $\mathcal{K}\subset L'\otimes\mathbb{R}$ be a real closed cone whose affine closure is top dimensional.
Assume that there exist $l'_* \in \mathcal{K}$ and a finite index sublattice $\widetilde{L}$ of $ L$ and a
quasipolynomial $\mathfrak{Q}^{{\mathcal K},S}_{h,\calv}(x)$
defined on $\widetilde{L}$ such that
\begin{equation}
\label{eq:qpol}
\mathfrak{Q}_{h,\calv}^{{\mathcal K},S}(l) = Q_{h,\calv}^{S}(r_h+l)
\end{equation}
whenever $l\in (l'_* +\mathcal{K})\cap \widetilde{L}$ ($r_h\in L'$ is defined similarly as in~\ref{ss:topol}).
Then we say that the counting function $Q_{h,\calv}^{S}$ (or $S_h(\mathbf{t})$)
\textit{admits a quasipolynomial} in $\mathcal{K}$, namely $\widetilde{L}\ni l\mapsto
\mathfrak{Q}_{h,\calv}^{\mathcal{K},S}(l)$.
In this case, we define the \textit{periodic constant} of $S_h(\bt)$ associated with $\mathcal{K}$ by
\begin{equation}\label{eq:PCDEF}
\mathrm{pc}^{\mathcal{K}}(S_h(\bt)) := \mathfrak{Q}_{h,\calv}^{{\mathcal K},S}(0) \in \mathbb{Z}.
\end{equation}
The number $\mathrm{pc}^{\mathcal{K}}(S_h(\bt))$ is independent of the
choice of $l'_*$ and of the finite index sublattice $\widetilde L\subset L$.

In some cases we might drop the indices $\mathcal{K}$ or $S$ if there is no ambiguity regarding them.

Given any $I\subset \calv$ the natural group homomorphism ${\rm pr}_I:L'\to L'_I$ preserves the lattices $L\to L_I$
and hence it induces a homomorphism $H\to H_I:=L'_I/L_I$, denoted by $h\mapsto h_I$. However, note that even if
$L'$ is the dual of $L$ associated with a form $(\,,\,)$,
$L_I'$ usually is not a dual lattice of $L_I$, it is just an overlattice. This fact also motivates the general setup of the present section.
In this projected context one can also define 
the periodic constant associated with the series $S_h(\bt_I)$ from the previous paragraph exchanging
$\calv$ (resp. $\bt$, $r_h$) by $I$ (resp. $\bt_I$, $(r_{h})_I$).
\begin{remark}
\label{rem:pc}
\mbox{}
\begin{enumerate}
 \item\label{rem:pc1}
 The periodic constant of one-variable series was introduced in~\cite{Ok,NOk} as follows.
 For simplicity, we assume that $L=L'\simeq\ZZ$ and let $S(t)=\sum_{l\geq 0}c_l t^l \in \mathbb{Z}[[t]]$ be a
 formal power series in one variable. If for some $p\in \mathbb{Z}_{>0}$ the counting function
 $Q^{(p)}(n):=\sum_{l=0}^{pn-1}c_l$ is a polynomial $\mathfrak{Q}^{(p)}$ in $n$, then the constant term
 $\mathfrak{Q}^{(p)}(0)$ is independent of $p$ and it is the periodic constant $\mathrm{pc}(S)$ of the series~$S$.
(Here $\widetilde{L}=p\ZZ$ and the cone is automatically the `positive cone'.)
 
 \item\label{rem:pc2}
 If $S(\bt)$ is a Laurent polynomial in $\ZZ[[L']]$, that is, $\Supp(S(\bt))$ is finite, then
 its counting function $Q_h^{S}$ is constant for large enough values and this constant
 equals the sum of the coefficients associated with the nonnegative exponents. Hence, if we denote by $S_h(\bt)|_{\nless 0}=\sum_{l'\in \Supp(S(\bt)),  l'\nless 0, [l'] = h} s(l')\bt^{l'}$, then $\mathfrak{Q}_h^{\, {\mathcal K},S}$ eg. for ${\mathcal K}=(\mathbb{R}_{\geq 0})^{|\calv|}$
 is the constant map $(S_h)|_{\nless 0}(\bf 1)$ (i.e. one substitutes for each $t_v=1$) and thus its periodic constant exists and equals~$(S_h)|_{\nless 0}(\bf 1)$.
(The periodic constant for ${\mathcal K}=(\mathbb{R}_{\leq 0})^{|\calv|}$ also exists and it equals zero.)

\item\label{rem:pc3}
Let's have again a lattice $L$ and an overlattice $L'$ with $H \cong L' / L$ and assume that $S(\bt)=\sum_{l' \in {\mathcal K}}a^S(l')\bt^{l'}$ is a series in
variables $l' \in L'$ supported on the cone ${\mathcal K}=\R_{\geq 0}\langle v_j\rangle_j\subset L\otimes \R$, where all the entries of each $v_j$ are positive. 
Assume that its counting function $Q_h^S(l')=\sum_{\tilde{l}\not\geq l', [\tilde{l}] = h }a^S(\tilde{l})$ admits the quasipolynomial $\mathfrak{Q}_h^S(l')$, which satisfies
 $\mathfrak{Q}_h^S(l')=Q_h^S(l')$ in a shifted cone of type $l_*+{\mathcal K}$.
 Then, in a convenient shifted cone,  for any fixed $l_0\in L$ one has
\begin{equation*}
 \mathfrak{Q}_h^S(l'+l_0)=\sum_{\tilde{l}\not\geq l'+l_0, [ \tilde{l}] = h}a^S(\tilde{l})=
\sum_{\tilde{l}\not\geq l', [ \tilde{l}] = h}a^{\bt^{-l_0}S}(\tilde{l}).
\end{equation*}
Let $\bt^{-l_0}S(\bt)|_{< 0}$ and $\bt^{-l_0}S(\bt)|_{\nless 0}$
be a decomposition of the Laurent series $\bt^{-l_0}S(\bt)$ according to its support.  

Then  $\bt^{-l_0}S(\bt)|_{\nless 0}$ is a Laurent series, while $\bt^{-l_0}S(\bt)|_{< 0}$ is a finite polynomial with negative exponents. Furthermore, by definition one has $Q_h^{\bt^{-l_0}S}(l')=Q_h^{\bt^{-l_0}S|_{\nless 0}}(l')$, and 
for $l$ with large  coefficients $\sum_{\tilde{l}\not\geq l', [ \tilde{l}] = h }a^{\bt^{-l_0}S|_{< 0}}(\tilde{l})=(\bt^{-l_0}(S_h)|_{< 0})({\bf 1})$. This shows that  for any $l_0\in L$ the series
$\bt^{-l_0}S_h(\bt)$ admits a quasipolinomial
and a periodic constant in the cone $\mathcal{K}$ and
$$\mathfrak{Q}_h^S(l_0)=(\bt^{-l_0}(S_h)|_{< 0})({\bf 1})+{\rm pc}^{\mathcal {K}}
(\bt^{-l_0}S_h).$$
 \end{enumerate}
\end{remark}

\subsection{Topological Poincar\'e series and Seiberg--Witten invariants}\label{ss:tps}
As we have promised previously, in this section we will define the topological Poincar\'e series associated with the topology of a complex normal surface singularity. We also present some important results which will be used throughout the article, regarding the intimate connection between its periodic constant and the Seiberg--Witten invariant of the link of the singularity.  Here, we consider the setting of section \ref{ss:topol}, thus the lattices $L\subset L'$ and the corresponding concepts are associated with a dual resolution graph $\calt$, or, equivalently, with the link $M$.

\subsubsection{\bf Definition of $Z(\bt)$}
The  \emph{multivariable topological Poincar\'e series} is the
Taylor expansion $Z(\mathbf{t})=\sum_{l'} z^\calt(l')\bt^{l'}
\in\bZ[[L']] $ at the  origin of the `zeta-function'
\begin{equation}\label{eq:1.1}
\prod_{v\in \mathcal{V}} (1-\mathbf{t}^{E^*_v})^{\delta_v-2},
\end{equation}
where
$\bt^{l'}:=\prod_{v\in \mathcal{V}}t_v^{l'_v}$  for any $l'=\sum _{v\in \mathcal{V}}l'_vE_v\in L'$ ($l'_v\in\bQ$).

It decomposes as $Z(\mathbf{t})=\sum_{h\in H}Z_h(\mathbf{t})$, where $Z_h(\mathbf{t})=\sum_{[l']=h}z^\calt (l')\bt^{l'}$. The expression
(\ref{eq:1.1}) shows that  $Z(\mathbf{t})$ is supported in the \emph{Lipman cone} $\mathcal{S}':=\mathbb{Z}_{\geq0}\langle E^{*}_{v}\rangle_{v\in\mathcal{V}}$.
Since  the intersection form $I$ is negative definite, all the entries of $E_v^*$ are strict positive, hence $\calS'\subset
\{\sum_vl'_vE_v\,:\, l'_v>0\}\cup\{0\}$.
Therefore, $Z(\bt)$ satifies the finiteness condition (\ref{eq:finiteness}) (cf. \cite[(2.1.2)]{NJEMS}) and following section \ref{ss:countingfunctions} one can consider its counting functions 
\begin{equation}\label{eq:countintro}
Q_h: L'_h:=\{x \in L'\,:\, [x]=h\}\to \bZ, \ \ \ \
Q_{h}(x)=\sum_{l'\ngeq x,\, [l']=h} z(l').
\end{equation}
for any $h\in H$. In the sequel, the notation $Q_h$ (without marking the associated series in upper index) will always stand for the counting function of $Z(\bt)$.

\subsubsection{\bf Seiberg--Witten invariants of $M$ as periodic constants of $Z(\bt)$}

We denote by $\mathrm{Spin}^c(\widetilde{X})$ the set of $spin^c$--structures on $\widetilde{X}$ and let $\widetilde{\sigma}_{can}$ be the {\it canonical
$spin^c$--structure on $\widetilde{X}$} identified by $c_1(\widetilde{\sigma}_{can})=Z_K$. $\mathrm{Spin}^c(\widetilde{X})$ is an $L'$-torsor and if we denote the action by $*$, then $c_1(l'*\widetilde{\sigma})=c_1(\widetilde{\sigma})+2l'$. All the $spin^c$--structures on $M$ are obtained by restrictions from $\widetilde{X}$. $\mathrm{Spin}^c(M)$ is an $H$--torsor, compatible with the projection $L'\to H$. The {\it canonical
$spin^c$--structure on $M$} is the restriction of $\widetilde{\sigma}_{can}$.

For further details regarding $spin^c$--structures the reader may consult with \cite[page 415]{GS}.

 We denote by $\mathfrak{sw}_{\sigma}(M)\in \bQ$ the
\emph{Seiberg--Witten invariant} of $M$ indexed by the $spin^c$--structures $\sigma\in {\rm Spin}^c(M)$ (cf. \cite{Lim, Nic04}, 
here we will use the sign convention of \cite{BN,NJEMS}.)

Then, an important combinatorial formula for $\{\mathfrak{sw}_\sigma(M)\}_\sigma$, developed  in \cite{NJEMS} using  the counting function of $Z(\mathbf{t})$, gives the following:

\begin{thm} \label{th:JEMS}\ \cite{NJEMS}
For any $l'\in Z_K+ \textnormal{int}(\mathcal{S}')$
\begin{equation}\label{eq:SUM} - Q_{[l']}(l')=
\frac{(-Z_K+2l')^2+|\mathcal{V}|}{8}+\mathfrak{sw}_{[-l']*\sigma_{can}}(M).
\end{equation}
\end{thm}
\noindent If we fix $h\in H$ and we write  $l'=l+r_{h}$ with $l\in L$,
then the right hand side of (\ref{eq:SUM}) is
 a multivariable quadratic polynomial on $L$. Therefore,  Theorem \ref{th:JEMS} can also be read in the following way:
 
\begin{thm}\label{th:NJEMSThm} \ \cite{NJEMS}
The counting function
 of $Z_h(\bt)$ in the cone $\calS'_{\mathbb{R}}:=\calS\otimes \mathbb{R}$
admits  the (quasi)polynomial
\begin{equation}\label{eq:SUMQP} \mathfrak{Q}_{h}(l)=
-\frac{(-Z_K+2r_h+2l)^2+|\mathcal{V}|}{8}-\mathfrak{sw}_{-h*\sigma_{can}}(M),
\end{equation}
whose periodic constant is
\begin{equation}\label{eq:SUMQP2}
\mathrm{pc}^{S'_{\mathbb{R}}}(Z_h(\mathbf{t}))=\mathfrak{Q}_h(0)=
-\mathfrak{sw}_{-h*\sigma_{can}}(M)-\frac{(-Z_K+2r_h)^2+|\mathcal{V}|}{8}.
\end{equation}
\end{thm}
The right hand side of (\ref{eq:SUMQP2}) 
will be called the  {\it normalized Seiberg--Witten invariant} of $M$ and for simplicity, in the sequel we will denote it by  
$\mathfrak{sw}^{norm}_h(\calt)$.

\subsubsection{\bf `Projected' version and surgery formula} \label{ss:surgform}

For the projected topological Poincar\'e series one has the following result:

\begin{theorem}\label{depth}\cite{LSz}
For a fixed $h \in H$ and a subset $I\subset \calv$ there is a unique quasipolynomial $\mathfrak{Q}_{h,I}(l)$, such that if $l' = r_h + l \in \sum_{v \in \calv} (\delta_v - 2) E_v^* + int(\calS')$ then 
\begin{equation}
\mathfrak{Q}_{h,I}(l) = Q_{h, I}(l') = \sum_{[l''] = [l'], l''_{I} \ngeq l'_{I}} z(l').
\end{equation}
\end{theorem}

In other words, this means that the projected topological Poincar\'e series $Z_{h, I}$ admits a unique periodic constant in the projected Lipman cone $\pi_{I}(S')$, which will be denoted by
$\mathrm{pc}_{h, \cali}$. Furthermore, a useful \emph{surgery formula} was developed in \cite{LNN} for this periodic constant, which will be presented in the following.

Consider a dual resolution graph $\calt$ and fix a subset $I\subset \calv$ of its vertices as above. 
The set of vertices $\calv \setminus I$ determines the connected full subgraphs $\{\calt_i\}_k$ with vertices
$\calv(\calt_i)$. We associate with any $\calt_i$ the lattices
$L(\calt_i)$ and $L'(\calt_i)$ as well, endowed with the corresponding intersection forms.

Then for each $i$ one considers the inclusion operator $j_i:L(\calt_i)\to L(\calt)$,
$E_v(\calt_i)\mapsto E_v(\calt)$, identifying naturally the corresponding $E$-base elements associated with the two
graphs. This preserves the intersection forms.

Let $j_{i}^*:L'(\calt)\to L'(\calt_i)$ be the dual (cohomological) operator, defined by
$j_{i}^*(E^*_{v}(\calt))=E^*_{v}(\calt_i)$ if $v\in\calv(\calt_i)$, and $j_{i}^*(E^*_{v}(\calt))=0$ otherwise.
Note that $j^*_{i}(E_v(\calt))=E_v(\calt_i)$ for any $v\in \calv(\calt_i)$.
Then we have the projection formula
$(j^*_{i}(l'), l)_{\calt_i}=(l',j_{i}(l))_{\calt}$ for any $l'\in L'(\calt)$ and $l\in L(\calt_i)$,
which also implies that
 $j^*_{i}(Z_K)=Z_{K}(\calt_i)$,
where $Z_K$ (resp. $Z_K(\calt_i)$) is the anti-canonical cycle associated with $\calt$ (resp. $\calt_i$).
%
%
%
%
%

Then, the  surgery formula on the quasipolynomial level is as follows: 

\begin{theorem}\label{surgery}\cite{LNN}
a)  If $l' = r_h + l , l \in L$ then 
\begin{align*}
\mathfrak{Q}_{h,\cali}(l) & = \left(-\mathfrak{sw}_{-h*\sigma_{can}}(M)-\frac{(-Z_K+2r_h + 2l)^2+|\mathcal{V}|}{8} \right) - \\ 
& -\sum_i \left(-\mathfrak{sw}_{-[j_i^*(r_h + l)]*\sigma_{can}}(M)-\frac{(-Z_K(\mathcal{T}_i)+2 j_i^*(r_h+ l))^2+|\mathcal{V}(\mathcal{T}_i)|}{8} \right).
\end{align*}

b) In particular, one expresses the projected periodic constant with the corresponding normalized Seiberg-Witten invariants as 
\begin{equation*}
\mathrm{pc}_{h, I} = \mathfrak{sw}_h^{norm}(\calt) - \sum_i \mathfrak{sw}_{[j_i^*(r_h)]}^{norm}(\calt_i).
\end{equation*}
\end{theorem}

Finally, we present a `duality' type result from \cite{LNNdual} which computes the periodic constant from certain values of the counting functions.

\begin{theorem}\label{dualcount}\cite{LNNdual}
Using the same notations as above we have the following formula:
\begin{equation}
\mathrm{pc}_{h, I} = Q_{[Z_K]-h, I}(Z_K-r_h)  = \sum_{[l'] = [Z_K]- h , l'_{\cali} \ngeq (Z_K - r_h)_{\cali}} z(l').
\end{equation}
\end{theorem}

\subsection{Splice quotient singularities}\label{ss:spq}

\subsubsection{} Splice quotient singularities were introduced by Neumann and Wahl in \cite{NWsq}. Associated with a resolution graph $\calt$ satisfying a semigroup and congruence condition they contructed an equisingular family of singularities whose analytic type is determined by these special properties. In the sequel, we use another approach to introduce splice quotient singularities which was developed by Okuma in \cite{Ok}. This is based on a \emph{monomial conditions} imposed on $\calt$ and which is equivalent with the  conditions of Neumann and Wahl. 

In the following, we fix a resolution graph $\calt$ and we distinguish the following subsets of vertices: a vertex $v$ is called \emph{node} if $\delta_v\geq 3$ and the set of nodes will be denoted by $\caln$; $v$ is an \emph{end-vertex} if $\delta_v=1$ and we will use $\cale$ for the set of end-vertices. For every node $v$, the connected components of $\calt\setminus v$ are called \emph{branches} of $v$. The set of end-vertices contained in a branch $\calt'$ is denoted by $\cale_{\calt'}$.

\begin{definition}{(\cite{Ok})} A cycle of the form $Z(a)=\sum_{v\in \cale}a_v E^*_v\in L'$ with $a_v\in \mathbb{Z}_{\geq 0}$ is called monomial cycle. We say that $\calt$ satisfies the \emph{monomial conditions} if for any node $v\in \caln$ and any branch $\calt'$ of $v$, there exists a monomial cycle $Z(a)$ such that $Z(a)-E^*_v$ is an effective integral cycle supported on $\calt'$. In paricular, one has $a_v=0$ for $v\notin \cale_{\calt'}$.  
\end{definition}
Associated with the resolution graph $\calt$ satisfying the monomial conditions one can construct the splice diagram equations, which define an isolated complete intersection singularity in $(\mathbb{C}^{|\cale|},0)$ on which the group $H$ acts freely off the origin. The singularity given by the factor is called \emph{splice quotient singularity} and has a good resolution with dual graph $\calt$. 

We remark that splice quotient singularities include eg. the rational singularities (with their arbitrary resolution graph) \cite{OkumaRat}, minimally elliptic singularities (with resolution graphs where the support of the minimally elliptic cycle is not proper) \cite{OkumaUAC} and weighted-homogeneous singularities (with their minimal good resolutions) \cite{neumannAC}. More details regarding the constructions and properties of splice quotient singularities can be found in \cite{NWsq,Ok,NCL}.

There is a third, more analytic condition which characterizes splice quotient singularities. This is the \emph{end-curve condition} which requires for a good resolution $\phi$ of $(X,0)$ with dual graph $\calt$ the existence of end-curve function for each component $E_v$, $v\in \cale$. This allows us to make the following definition for splice quotient singularities, cf. \cite{NCL, OECTh, NWECTh}.

\begin{defn}
A singularity $(X,0)$ with a good resolution $\phi:\widetilde{X}\to X$ and dual graph $\calt$ is called spliced quotient if for every end-vertex $v \in \cale$ one has $H^0(\calO_{\tX}( -E_v^*))_{reg} \neq \emptyset$.
\end{defn}

\subsubsection{\bf Analytic vs. topological invariants for splice quotient singularities} \label{ss:sqat}

First of all,  by \cite{NOk} we know that if $(X,0)$ is a splice quotient singularity with good resolution $\phi:\widetilde{X}\to X$ and dual graph $\calt$ , then the normalized Seiberg-Witten invariants are expressed with cohomology numbers as follows: 
$$ \mathfrak{sw}_h^{norm}(\calt) = h^1(\calO_{\tX}(-r_h)).$$ 
For $h=0$, this gives the equivalence of the geometric genus of $(X,0)$ with the canonical normalized Seiberg-Witten invariant, or, equivalently, with the periodic constant of $Z_0(\bt)$. In particular, this also implies for splice quotient singularities  the positivity of the canonical Seiberg-Witten invariant $\mathfrak{sw}^{norm}_0(\calt)\geq 0$.

%


The previous identification can be generalized to the level of series as well. Namely,  for a normal surface singularity $(X, 0)$ with a fixed resolution and resolution graph $\mathcal{T}$ one can associate with it the analytic Poincaré series $ P(\mathbf{t}) = \sum_{l'} p(l')\bt^{l'} \in\bZ[[L']] $ defined by Campillo, Delgado and Gusein-Zade \cite{CDGPs,CDGEq}.  The original definition is based on $H$-equivariant $L'$-indexed divisorial filtration on the local ring $\calO_Y$, where $Y$ is the universal abelian covering of $(X,0)$. However, one can interpret its  coefficients using the formula $p(l') = \sum_{I \subset \calv} (-1)^{|I|} h^0(\calO_{Z-E_I}(-l' - E_I)$ where $Z$ is an enough large integral cycle. 

Then the main result from \cite{NCL}, valid for splice quotient singularities gives us the following identification:
$$Z(\mathbf{t}) = P(\mathbf{t}).$$

Finally, as a particular consequence of the above equality, we mention the following formula which will be used later in the article. Namely, if we consider a resolution graph $\mathcal{T}$ satisfying the monomial conditions, its associated  splice quotient singularity and a fixed cycle $l'\in L'$, then one has the identity:
\begin{equation}\label{eq:sqcount}
\sum_{l \in L, l \geq 0, l_v = 0}z^{\mathcal{T}}(l' + l) = \dim\left(\frac{H^0(\tX, \calO_{\tX}( -l'))}{H^0(\tX, \calO_{\tX}( -l'-E_v))}\right).
\end{equation}
Therefore we get that
$0 \leq \sum_{l \in L, l \geq 0, l_v = 0}z^{\mathcal{T}}(l' + l)  \leq -(l', E_v) + 1$ by estimating the right hand side of equation\ref{eq:sqcount} with the standard exact sequence.

\section{Motivation: computations in the case of splice quotient singularities}\label{s:motspq}

First of all, we recall the following result from \cite{NNA1}.

\begin{theorem}\cite{NNA1}\label{thm:gat}
Let us fix a resolution $\tX$ of a normal surface singularity $(X, 0)$ and a very large cycle $Z$ on it. Consider a cycle $l' \in L'$ and a line bundle $\calL \in \pic^{l'}(Z)$ 
and an arbitrary point $p \in (c^{l'}(Z))^{-1}(\calL) \subset \eca^{l'}(Z)$. We denote by $T_p$ the tangent space of the smooth variety $\eca^{l'}(Z)$ at the point $p$  and by $F_p$ the tangent space of the smooth fiber $(c^{l'}(Z))^{-1}(\calL)$ at the same point. Then  the tangent map $T_{c^{l'}(Z)} : T_p / F_p \to \bC^{p_g}$ is injective and the image has  codimension $h^1(\tX, \calL)$.

Assume further that $p$ consists of disjoint transversal cuts. Then we have $h^1(\tX , \calL) = p_g(\tX) - \dim( T_{c^{l'}}(T_p))$. In particular $h^1(\tX , \calL)$ equals the dimension of the differential forms in $\frac{H^0(\calO_{\tX}(K + [Z_K] ))}{H^0(\calO_{\tX}(K))}$
having no pole on the divisor $p$.
\end{theorem}

Now we consider a resolution graph $\mathcal{T}$ which satisfies the monomial conditions, and let $\tX$ be the  splice quotient analytic type on it. 
Consider a cycle $l' \in L'$ such that for all vertices $v \in |l'|^*$ we have $E_v^* \in S'_{an}$ (for example all vertices $v \in |l'|^*$ are ends or nodes).
Suppose futhermore that we have a large cycle $Z$ and a divisor $p \in (c^{l'}(Z))^{-1}(0) \subset \eca^{l'}(Z)$ such that $p$ consists of disjoint transversal cuts, all of them defining natural line bundles.

In the following we investigate what the Theorem \ref{thm:gat} gives combinatorially in the above special situation.

Now let $l' = \sum_{v \in \calv} a_v E_v^*$ and $s = \{s_{v, k} \geq 0| 1 \leq k \leq a_v \}$ be the set of very large positive integers. 

We blow up the divisor $E_v$  sequentially in $ a_v$ points, $s_{v, k}$ times in each point, such that we always blow up at an intersection point of the strict transform of the divisor $p$ and the new exceptional divisor. Let us denote the new end-vertices of the new resolution $\tX_{s}$ by $v_k, 1 \leq k \leq a_v$ and the strict transform of $p$ by $p_s$. (We also distinguish by putting the lower index $s$ for all the objects associated with $\tX_s$.)

Let's denote the connected components of the divisor $p_s$ by $p_{v, k, s}$.

One can easily identify the differential forms in $\frac{H^0(\calO_{\tX}(K + [Z_K] ))}{H^0(\calO_{\tX}(K))}$ which has no pole on the divisor $p$ with the differential forms in $\frac{H^0(\calO_{\tX_s}(K_s + [Z_{K, s}] ))}{ H^0(\calO_{\tX}(K_s))}$ vanishing of order at least $s_{v, k}$ on the divisor  $p_{v, k, s}$.

The vanishing criteria is well-defined, since every differential form in $H^0(\calO_{\tX}(K_s))$ vanish at least of order $s_{v, k}$ on the divisor $p_{v, k, s}$, and the identification is straightforward.

We know that the singularity we get after all the blow ups is splice qoutient, in particular it is $\mathbb{Q}$-Gorenstein. 
This means that $\frac{H^0(\calO_{\tX_s}(K_s + [Z_{K, s}] ))}{ H^0(\calO_{\tX}(K_s))}$ can be identified with $\frac{H^0(\calO_{\tX_s}(-Z_{K, s}+ [Z_{K, s}] ))}{H^0(\calO_{\tX}(-Z_{K, s}))}$.

Notice that if the numbers $s_{v, k}$ are large enough and some $w \in \frac{H^0(\calO_{\tX_s}(-Z_{K, s}+ [Z_{K, s}] ))}{H^0(\calO_{\tX}(-Z_{K, s}))}$ vanishes of order $s_{v, k}$ on the divisor $p_{v, k, s}$ but does not vanish on the exceptional divisor $E_{v_{k, s}}$ of order $s_{v, k}$, then there is a cut of the differential form going through the intersection point of  $p_{v, k, s}$ and $E_{v_{k, s}}$.         

Indeed let's take a representative of $w$, say $\omega \in  H^0(\calO_{\tX_s}(-Z_{K, s}+ [Z_{K, s}] ))$ and consider its divisor $div(\omega) = l_{\omega} + C$, where $l_{\omega}$ is supported on $E_{s}$ and $C$ is containing the cuts of $\omega$.

Let's denote by $c_1(C)$ the Chern class of the line bundle $\calO_{\tX_s}(C)$.

Notice that  $l_{\omega} + c_1(C) = -Z_{K, s}+ [Z_{K, s}]$ and $c_1(C)_{v_{k, s}} \leq (-E_{v_{k, s}}^*)_{v_{k, s}} = s_{v, k} + (-E_{v}^*)_v$ and $(-Z_{K, s})_{v_{k, s}}
= (-Z_K)_v - s_{v, k}$. Therefore if $s_{v, k}$ is enough large then $(l_{\omega})_{v_{k, s}}> s_{v, k} + [Z_{K, s}]$. Thus, in fact $w$ vanishes of order $s_{v, k}$ along the exceptional divisor $E_{v_{k, s}}$ which is a contradiction.

This means that if the number $s_{v, k}$ are large enough then a differential form $w \in \frac{H^0(\calO_{\tX_s}(-Z_{K, s}+ [Z_{K, s}] ))}{H^0(\calO_{\tX}(-Z_{K, s}))}$ vanishes of order $s_{v, k}$ on the divisor $p_{v, k, s}$ if and only if it vanishes along the exceptional divisor $E_{v_{k, s}}$ of order $s_{v, k}$.

Hence, we have the following equality:

\begin{equation}
h^1(\calO_{\tX}(-l')) = \dim \left( \frac{H^0\Big(\calO_{\tX_{s}}( \sum_{u \neq v_{k,s}}(- Z_{K, s} +[Z_{K, s}])_u E_u - \sum_{v \in \calv, 1 \leq k \leq a_v}  -(Z_K)_v  E_{v_{k, s}}  \Big)}{ H^0(\calO_{\tX_{s}}( -Z_{K, s}) ))} \right).
\end{equation}

Since the singularity $\tX_s$ is splice quotient, hence $\bQ$-Gorenstein, we can rewrite this formula using the topological Poincar\'e series $Z^{\mathcal{T}_s}(\bt) = \sum_{l'' \in S'_s}z^{\mathcal{T}_s}(l'') \bt^{l''}$associated with the newly created resolution graph $\calt_s$ as follows.

\begin{equation}\label{hegymaskepp}
h^1(\calO_{\tX}(-l'))  = \sum_{[l''] = [Z_{K, s}], l'' \geq  \sum_{v \in \calv, 1 \leq k \leq a_v}  (Z_K)_v  E_{v_{k, s}} , l'' \ngeq \pi_s^*(Z_K) }  z^{\mathcal{T}_s}(l'') . 
\end{equation}

In the following, let $S=\{p_{v, k}\}$be the set of the intersection points of the divisor $p$ with the exceptional divisor $E=\cup_v E_v \subset \tX$  and for every element $x \in S$ denote by $v_x$ the corresponding intersected component $E_v$ .

With an arbitrary subset $I \subset S$ we associate the cycle $l'(I) = \sum_{x \in I} E_{v_x}^*$ and let's denote $|l'|^* = J$.  
Furthermore, consider the following multivariable series $$R(\mathbf{t}) = \prod_{ u\in \calv}(1-\mathbf{t}^{E_u^*})^{\delta_u -2 + a_u}=\sum_{l''}r(l'')\bt^{l''}.$$

Then, as it is anounced in the introduction as Theorem (C), we claim the following formula.

\begin{theorem}\label{dualityspq}
\begin{equation*}
\sum_{I \subset S} (-1)^{|I|} h^1(\calO_{\tX}( - l'(I))) = \sum_{l'' \in S',[ l''] = [Z_K], l''_J \leq (Z_K-E)_J} r(l'').
\end{equation*}
\end{theorem}

\begin{proof}
For an element $x \in S$ we denote by $C_x$ the corresponding cut and by by $C'_x$ its strict transform along the blow up $\pi_s$.

For a fixed subset $I \subset S$ we blow up $\tX$ at the cut $C_x$ $s_x$-many times if $x \in I$, and we denote the blow up map by $\pi_I : \tX_I \to \tX$. We know that $\tX_s$ is a further blow up of $\tX_I$, and the corresponding blow up map will be $\pi_s: \tX_s \to \tX_I$.  
If $x \notin I$ we use the same notation $C_x$ for the strict transform of $C_x$. 

Then we know that $ h^1(\calO_{\tX}( - l'(I))) = h^1(\calO_{\tX_{S \setminus I}}( - \pi_{S \setminus I} ^*(l'(I)))) = h^1(\calO_{\tX_{S \setminus I}}( \sum_{x \in I} C_x))$.

Now, we can use equation (\ref{hegymaskepp}) for the splice quotient singularity $\tX_{S \setminus I}$ and the line bundle $\calO_{\tX_{S \setminus I}}( - \pi_{S \setminus I} ^*(l'(I)) = \calO_{\tX_{S \setminus I}}( \sum_{x \in I} C_x)$ in order to get

\begin{equation*}
h^1(\calO_{\tX}( - l'(I))) =  \sum_{[l''] = [Z_{K, s}], l'' \geq  \sum_{v \in \calv, 1 \leq k \leq a_v,  p_{v, k} \in I}  (Z_K)_v  E_{v_{k, s}} , l'' \ngeq \pi_s^*(Z_K) }  z^{\mathcal{T}_s}(l'').
\end{equation*} 

It means that we have:

\begin{equation}\label{h1}
\sum_{I \subset S} (-1)^{|I|} h^1(\calO_{\tX}( - l'(I))) =  \sum_{[l''] = [Z_{K, s}], l''_{v_{k, s}} < (Z_K)_v , l'' \ngeq \pi_s^*(Z_K) }  z^{\mathcal{T}_s}(l'').
\end{equation} 

Notice that if $ z^{\mathcal{T}_s}(l'') \neq 0$, then $|l''|^* \subset \calv \cup_{v, 1 \leq k \leq a_v} v_{k, s}$.  However if there is a vertex $v$ and number $1 \leq k \leq a_v$
such that $v_{k, s} \in |l''|^*$ then we would have $l''_{v_{k, s}} > (Z_K)_v$ since $(E_{v_{k, s}}^*)_{v_{k, s}}$ is very large.

Hence, this induces the desired formula

\begin{equation}\label{dualityline}
\sum_{I \subset S} (-1)^{|I|} h^1(\calO_{\tX}( - l'(I))) = \sum_{l' \in S',[ l''] = [Z_K], l''_J \leq (Z_K-E)_J} r(l'').
\end{equation}

\end{proof}

Note that by setting the notation $h_I = [l'(I)]$ we can express the cohomology number as 

\begin{equation*}
h^1(\calO_{\tX}( - l'(I))) =  -\sum_{l'' \ngeq l'(I),[ l'] = h_I}  z(l') + \chi(l'(I)) + h^1(\calO(-r_{h_I})) -\chi(r_{h_I}).
\end{equation*}
Therefore, (\ref{dualityline}) can also be write in the following way

\begin{equation}\label{dualitycounting}
\sum_{I \subset S} (-1)^{|I|} \left(   -\sum_{l'' \ngeq l'(I),[ l'] = h_I}  z(l') + \chi(l'(I)) + h^1(\calO(-r_{h_I})) -\chi(r_{h_I}) \right) = \sum_{l'' \in S',[ l''] = [Z_K], l''_J \leq (Z_K-E)_J} r(l'').
\end{equation}

\section{Virtual cohomology number and topological duality for counting functions} \label{s:topdcf}
In this section we discuss and prove the identity (\ref{dualityline}) generally, in the topological setting. As a starting point, motivated by the equation (\ref{h1}) we define a topological canditate for the cohomology numbers which will also help us to present the formulae in a convenient way.

\subsection{Virtual cohomology number}

Let $\calt$ be a good resolution graph of a normal surface singularity whose link is a rational homology sphere and consider its topological Poincare series $Z(\mathbf{t}) = \prod_{v \in \calv} (1-\mathbf{t}^{E_v^*})^{\delta_v -2}$, cf. section \ref{ss:tps}. Then one considers the following definition:

\begin{defn}
For any cycle $l'\in L'$ we define the \emph{virtual cohomology number} as 
\begin{equation}
h^1_{virt}(l') = - Q_h(l')  + \chi(l') - \chi(r_h) + \mathfrak{sw}_{h}^{norm}(\calt),
\end{equation}
where we recall that $h = [l']$, $Q_h(l')=\sum_{[l''] = h, l'' \ngeq l'} z(l'')$ is the counting function of $Z(\bt)$ and $\mathfrak{sw}_{h}^{norm}(\calt)$ is the normalized Seiberg-Witten invariant. 
\end{defn}

\begin{remark}\label{rem:h1virt}
(1) \ In particular, if $\calt$ satisfies the monomial conditions and admits a splice quotient singularity, then  $h^1_{virt}(l')$ equals with the standard cohomology number $h^1(\calO_{\tX}(-l'))$.

(2) \ If $l'\in Z_K+int(\calS')$ then by \cite{NJEMS} one knows that $Q_h(l')=\chi(l') - \chi(r_h) + \mathfrak{sw}_{h}^{norm}(\calt)$. Hence $h^1_{virt}(l')=0$, which corresponds to the Grauert-Riemenschneider vanishing (cf. \cite{GrRie}) in the analytic setting. 

Therefore, in general the virtual cohomology number measures exactly the difference between the counting function and its quasipolynomial associated with the cone $\calS'$.

(3) \ By substituting $l'=r_h$ for some $h\in H$ yields $h^1_{virt}(r_h)=\mathfrak{sw}^{norm}_h(\calt)$, since $Q_h(r_h)=0$ by the definition of the counting function.  

\end{remark}

\subsection{The `duality'} We fix a cycle $l'\in L'$ and express it in the $E^*$-basis as $l'=\sum_{v\in \calv}a_v E^*_v$. For the ease of notations, if there is no ambiguity we will write simply $J:=\{v\in \calv : a_v\neq 0\}$ for the $E^*$-support of $l'$, $|l'|^*$.

Then consider the multiset $J^m:=\{v^{a_{v}} : v\in J\}$, where in $J^m$ each vertex $v\in J$ has a multiplicity $a_v$. Let $I^m\subset J^m$ be a non-trivial multisubset, $I^m=\{v_{i_1}^{k_{i_1}},\dots, v_{i_r}^{k_{i_r}}\}$, where $\{v_{i_1},\dots,v_{i_r}\}=I\subset J$  are different vertices from  $J$ with multiplicity $0<k_{i_j}\leq a_{v_{i_j}}$ for any $j$, and $0<r\leq |J|$  (Notice that the numbers $r$, $k_{i_j}, 1 \leq j \leq r$ do not determine the multisubset, it is also important which $k_{i_j}$ elements we pick out from $v_{i_j}^{a_{v_{i_j}}}$).

 For such a sub-multiset $I^m$ we define $|I^m|:=\sum_j k_{i_j}$ and the cycle
$$l'(I^m):=\sum_{j} k_{i_j} E^*_{v_{i_j}}.$$ 
Furthermore, associated with a cycle $l'$ we define a `relative' type  series as the Taylor expansion at the origin $\sum_{l''\in \calS'}r(l'')$ of the multivariable function 
$$R(\mathbf{t}) = Z(\mathbf{t}) \cdot \prod_{ v\in J}(1-\mathbf{t}^{E_v^*})^{a_v}.$$  
Then, the main result of this section is as follows:

\begin{theorem}\label{topdualcount} 
If $\calt$ is a good resolution graph of a normal surface singularity whose link is a rational homology sphere, then for any cycle $l'\in L'$ we have the following identity
\begin{equation}\label{dualitycounting2}
\sum_{I^m \subset J^m}  (-1)^{|I^m|} \cdot h^1_{virt}( l'(I^m)) = \sum_{l'' \in \calS',[ l''] = [Z_K] \atop l''_J \leq (Z_K-E)_J} r(l'').
\end{equation}
\end{theorem}


\subsection{}
Before we proceed to the proof, we need the following result from \cite{LSz} and one of its natural consequence.
\begin{lemma}\cite{LSz}\label{redlaurent}
Let $\overline{v' v''}$ be an edge of a tree $\mathcal{T}$ with vertex set $\calv$. Decompose $\mathcal{T} \setminus \overline{v' v''}$ into disjoint union of trees
$\mathcal{T}_{\calv'}$ and $\mathcal{T}_{\calv''}$ with vertex sets $v' \in \calv'$ and $v'' \in \calv''$ respectively. Then for a non-empty vertex set $J \subset \calv''$ we have:
\begin{equation*}
(1- \bt_J^{E_{v''}^*}) \cdot \prod_{v \in \calv'} (1-\bt_J^{E_{v}^*})^{\delta_{v}-2}  \in \bZ[\bt_J].
\end{equation*}
So the fraction is a Laurent polynomial, supported on $\pi_{J}(\calS')$.
\end{lemma}
Then, one can prove the following corrolary of this lemma.

\begin{cor}\label{laurent}
Let $\mathcal{T}$ be a tree with vertex set $\calv$ and consider a non-empty subset $\emptyset\neq J \subset \calv$. Then we have 
\begin{equation*}
\prod_{v \in J}(1- \bt_J^{E_{v}^*}) \cdot \prod_{v \in \calv} (1- \bt_J^{E_{v}^*})^{\delta_{v}-2}  \in \bZ[\bt_J].
\end{equation*}
So the fraction is a Laurent polynomial, supported on $\pi_{J}(\calS')$.
\end{cor}

\begin{proof}
Let $\mathcal{T}'$ be the smallest subtree of $\mathcal{T}$ such that its set of vertices $\calv'$ contains $J$. The valency of a vertex $v\in \calv'$ in the subgraph $\calt'$ will be denoted by $\delta_{v,\calv'}$. Then we denote the connected components of $\mathcal{T} \setminus \calv'$ by $\mathcal{T}_1, \cdots, \mathcal{T}_n$, and their set of vertices by $\calv_1,\dots,\calv_n$, respectively.

Since $\mathcal{T}$ is a tree, each component $\mathcal{T}_i$ ($1 \leq i \leq n$) is connected to a unique vertex in $\calv'$, which will be denoted by $v_i$. Then one can write
\begin{equation*}
 \prod_{v \in \calv} (1-\bt_J^{E_{v}^*})^{\delta_{v}-2}  = \prod_{v \in \calv'} (1- \bt_J^{E_{v}^*})^{\delta_{v, \calv'}-2} \cdot \prod_{1 \leq i \leq n} \left( \prod_{v \in \calv_i} (1- \bt_J^{E_{v}^*})^{\delta_{v}-2} \cdot (1- \bt_J^{E_{v_i}^*}) \right) .
\end{equation*} 

Notice that by Lemma \ref{redlaurent} for each $1 \leq i \leq n$ we have $\prod_{v \in \calv_i} (1-\bt_J^{E_{v}^*})^{\delta_{v}-2} \cdot (1- \bt_J^{E_{v_i}^*}) \in \bZ[\bt_J]$. On the other hand, one also has $\prod_{v \in J}(1- \bt_J^{E_{v}^*}) \cdot \prod_{v \in \calv'} (1- \bt_J^{E_{v}^*})^{\delta_{v, \calv'}-2} \in \bZ[\bt_J]$ because the end-vertices of the subgraph $\mathcal{T}'$ are in $J$. These prove the statement completely.

\end{proof}

\subsection{Proof of Theorem \ref{topdualcount}}

We consider the following `twisted' relative Poincar\'e series associated with $l'$ as the Taylor expansion at the origin of the rational function 
$$R'(\mathbf{t}) = Z(\mathbf{t}) \cdot \prod_{ u\in J}(1-\mathbf{t}^{E_u^*})^{a_u} \cdot \mathbf{t}^{-Z_K}.$$ 
For simplicity, we denote both the rational function and its Taylor expansion at the origin with the same symbol. We look at the $h=0$-part of $R'(\mathbf{t})$, which can be expressed as  
\begin{equation}\label{eq:R'0}
R'_0(\mathbf{t}) = \sum_{I^m \subset J^m}  (-1)^{|I^m|} \mathbf{t}^{-(Z_K - l'(I^m))} \cdot Z_{[(Z_K - l'(I^m))]}(\mathbf{t}),
\end{equation}
and consider the projected series $\pi_{J}(R'_0) = R'_{0, J}$. Then the idea is to compute the periodic constant $\mathrm{pc}^{\pi_{J}(S'_{\mathbb{R}})}(R'_{0, J}(\mathbf{t}_{J}))$ of $R'_{0, J}$ in two different ways as follows.
\vspace{0.2cm}

I. \ First of all, by (\ref{eq:R'0}) and the linearity of the periodic constant 
we have that 
$$\mathrm{pc}^{\pi_J(S'_{\mathbb{R}})}(R'_{0, J}(\mathbf{t}_J)) = \sum_{I^m \subset J^m_{l'}}  (-1)^{|I^m|} \mathrm{pc}^{\pi_J(S'_{\mathbb{R}})}( \mathbf{t}_J^{-(Z_K - l'(I^m))} \cdot Z_{[(Z_K - l'_I)], J}(\mathbf{t}_J)).$$ Then  the periodic constants appearing in the summation on the right hand side of the previous identity can be expressed in the following way.

Let $Q_{[(Z_K - l'(I^m))],J}(l'')$ be the counting function of $Z_{[(Z_K - l'(I^m))], J}(\mathbf{t}_J)$, and let $Q'_{0,J}(l)$ be the counting function of the Laurent series $ \mathbf{t}_J^{-(Z_K - l'(I^m))} \cdot Z_{[(Z_K - l'(I^m))], J}(\mathbf{t}_J)$. Then,  by definitions from section \ref{ss:pc} we have: 

\begin{equation*}
Q'_{0, J}(l)  =  Q_{[Z_K - l'(I^m)],J}(l + Z_K -l'(I^m)) -  \sum_{l''_J \leq (Z_K - l'(I^m)-E)_J \atop [l''] = [(Z_K - l'(I^m))]} z(l'') .
\end{equation*}
This implies the same identity on the level of quasipolynomials, that is $$\mathfrak{Q}'_{0, J}(l)  =  \mathfrak{Q}_{[Z_K - l'(I^m)],J}(l + Z_K -l'(I^m)) -  \sum_{l''_J \leq (Z_K - l'(I^m)-E)_J \atop [l''] = [(Z_K - l'(I^m))]} z(l'') ,$$ where $\mathfrak{Q}'_{0, J}$ is the quasipolynomial associated with $Q'_{0, J}$. 
Since $\mathrm{pc}^{\pi_J(S'_{\mathbb{R}})}( \mathbf{t}_J^{-(Z_K - l'(I^m))} \cdot Z_{[(Z_K - l'(I^m))], J}(\mathbf{t}_J))$ equals the constant term of $\mathfrak{Q}'_{0,J}(l)$, one gets 

\begin{equation*}
\mathrm{pc}^{\pi_J(S'_{\mathbb{R}})}( \mathbf{t}_J^{-(Z_K - l'(I^m))} \cdot Z_{[(Z_K - l'_I)], J}(\mathbf{t}_J)) =    \mathfrak{Q}_{[Z_K - l'(I^m)],J}(Z_K -l'(I^m)) -  \sum_{l''_J \leq (Z_K - l'(I^m)-E)_J \atop [l''] = [(Z_K - l'(I^m))]} z(l'') .
\end{equation*}

Let us denote the connected components of the graph $\calt \setminus J$ by $\mathcal{T}_1,\dots,\calt_k$ and consider the dual operator $j_i^*:L'\to L'(\calt_i)$ as in section \ref{ss:surgform}.  We write $Z_K^{\calt_i}$ for the corresponding  anti-canonical cycles. Then the surgery formula from Theorem \ref{surgery} gives
\begin{equation}\label{eq:surgspec}
\mathfrak{Q}_{ [Z_K - l'(I^m)],J}(Z_K - l'(I^m)) + \sum_{i=1}^k \mathfrak{Q}^{\mathcal{T}_i}_{[Z_K^{\calt_i}]}(Z_K^{\calt_i}) = \mathfrak{Q}_{[Z_K - l'(I^m)]}(Z_K - l'(I^m)).
\end{equation}
On the other hand, according to Remark \ref{rem:h1virt}(3) one can write $\mathfrak{Q}_{[Z_K - l'(I^m)]}(Z_K - l'(I^m)) =  \chi(Z_K - l'(I^m)) + h^1_{virt}(r_{[Z_K - l'(I^m)]}) - \chi(r_{[Z_K - l'(I^m)]})$ which, together with (\ref{eq:surgspec}) yields the following formula:
\begin{align*}
\mathrm{pc}^{\pi_J(S'_{\mathbb{R}})}( \mathbf{t}_J^{-(Z_K - l'(I^m))} \cdot Z_{[(Z_K - l'_I)], J}(\mathbf{t}_J))  &=  \chi(Z_K - l'(I^m)) +  h^1_{virt}(r_{[Z_K - l'(I^m)]}) - \chi(r_{[Z_K - l'(I^m)]}) \\
&- \sum_{i=1}^k \mathfrak{Q}^{\mathcal{T}_i}_{[Z_K^{\calt_i}]}(Z_K^{\calt_i}) - \sum_{l''_J \leq (Z_K - l'(I^m)-E)_J \atop [l''] = [(Z_K - l'(I^m))]} z(l'').
\end{align*}

Note that by Remark \ref{rem:h1virt}(3), $h^1_{virt}(r_{[Z_K - l'(I^m)]})$ is known and it is the normalized Seiberg-Witten invariant $\mathfrak{sw}_{[Z_K - l'(I^m)]}^{norm}(\calt)$. 
Therefore, one can use the $Z_K$-symmetrical identitites  $ \chi(Z_K -l'(I^m)) =   \chi(l'(I^m))$ and the identification $$ \mathfrak{sw}_{[Z_K - l'(I^m)]}^{norm}(\calt) - \chi(r_{[Z_K - l'(I^m)]}) =  \mathfrak{sw}_{[ l'(I^m)]}^{norm}(\calt) - \chi(r_{[ l'(I^m)]}) $$ (cf. \cite{LNNdual}) in order to get 

\begin{align*}
\mathrm{pc}^{\pi_J(S'_{\mathbb{R}})}( \mathbf{t}_J^{-(Z_K - l'(I^m))} \cdot Z_{[(Z_K - l'_I)], J}(\mathbf{t}_J))  &=  \chi(Z_K - l'(I^m)) +  \mathfrak{sw}_{[ l'(I^m)]}^{norm}(\calt) - \chi(r_{[l'(I^m)]}) \\
&- \sum_{i=1}^k \mathfrak{Q}^{\mathcal{T}_i}_{[Z_K^{\calt_i}]}(Z_K^{\calt_i}) - \sum_{l''_J \leq (Z_K - l'(I^m)-E)_J \atop [l''] = [(Z_K - l'(I^m))]} z(l'').
\end{align*}

Therefore, the first way of our periodic constant computation deduces the following:

\begin{equation}\label{periodic1}
\mathrm{pc}^{\pi_J(S'_{\mathbb{R}})}(R'_{0, J}(\mathbf{t}_J)) =  \sum_{I^m \subset J^m_{l'}}  (-1)^{|I^m|} \cdot \Big(   \chi(l'(I^m)) +   \mathfrak{sw}_{[ l'(I^m)]}^{norm}(\calt) - \chi(r_{[l'(I^m)]}) - \sum_{l''_J \leq (Z_K - l'(I^m)-E)_J \atop [l''] = [(Z_K - l'(I^m))]} z(l'') \Big)
\end{equation}
\vspace{0.2cm}

II. The second approach uses the fact that $R'_{0, J}(\mathbf{t}_J)$ is a Laurent polynomial, implied by Corollary \ref{laurent}. Hence one gets immediately that $\mathrm{pc}^{\pi_J(S'_{\mathbb{R}})}(R'_{0, J}(\mathbf{t}_J)) = \sum_{ [l] = 0, l_J \nleq -E_J} r'(l)$. 

We know that $R'(\mathbf{t})= \sum_{I^m \subset J^m_{l'}}  (-1)^{|I^m|} \mathbf{t}^{-(Z_K - l'(I^m))} \cdot Z(\mathbf{t})$. Let $$T^{\infty}(\mathbf{t}^{-(Z_K - l'(I^m))} \cdot Z(\mathbf{t}))=\sum_{l''} r'_{I^m,\infty}(l'')$$ be the Taylor expansion at infinity of the corresponding function, hence $T^{\infty}(R'(\mathbf{t}))= \sum_{I^m \subset J^m_{l'}}  (-1)^{|I^m|} \cdot T^{\infty}(\mathbf{t}^{-(Z_K - l'(I^m))} \cdot Z(\mathbf{t}))$. Since  $R'_{0, J}(\mathbf{t}_J)$ is a Laurent polynomial, we can also  express its periodic constant as  
\begin{equation}\label{eq:pcneg}
\mathrm{pc}^{\pi_J(S'_{\mathbb{R}})}(R'_{0, J}(\mathbf{t}_J)) = \sum_{I^m \subset J^m_{l'}}  (-1)^{|I^m|} \cdot   \sum_{ [l] = 0, l_J \nleq (-E)_J} r'_{I^m,\infty}(l).
\end{equation}

Furthermore, one has the following symmetry  
$$T^{\infty}(\mathbf{t}^{-(Z_K - l'(I^m))} \cdot Z(\mathbf{t}))=\mathbf{t}^{-(Z_K - l'(I^m))}\cdot T^{\infty}(Z(\mathbf{t})))=\mathbf{t}^{-(Z_K - l'(I^m))}\cdot \mathbf{t}^{Z_K - E} \cdot Z(\mathbf{t}^{-1}), $$
 see eg. \cite[(4.4.2)]{LNNdual}, which gives the equalities 
 $$\sum_{ [l] = 0, l_J \nleq (-E)_J}r'_{I^m,\infty}(l)  = \sum_{[l''] = l'(I^m)], l''_J \ngeq (l'(I^m))_J} z(l'')=\sum_{[l''] = l'(I^m)], l'' \ngeq (l'(I^m))} z(l''),$$ 
 where the second follows from the fact that if $l'' \in S'$ and $l'' \ngeq l'(I^m)$ then $ l''_J \ngeq  l'(I^m)_J$. These transform (\ref{eq:pcneg}) into 
 \begin{equation}\label{periodic2}
\mathrm{pc}^{\pi_J(S'_{\mathbb{R}})}(R'_{0, J}(\mathbf{t}_J))  = \sum_{I^m \subset J^m_{l'}}  (-1)^{|I^m|} \cdot \Bigg( \sum_{[l''] = [l'(I^m)], l'' \ngeq l'(I^m)} z(l'') \Bigg).
\end{equation}

Finally, by combining the two formulae (\ref{periodic1}) and (\ref{periodic2}), and arranging the terms in a convenient way we deduce the following equality:

\begin{align}\label{eq:finaleq1}
\sum_{I^m \subset J^m_{l'}}  (-1)^{|I^m|} \cdot \Big(   \chi(l'(I^m)) - \chi(r_{[l'(I^m)]})+\mathfrak{sw}_{[ l'(I^m)]}^{norm} - \sum_{[l''] = [l'(I^m)], l'' \ngeq l'(I^m)} z(l'')  \Big)\\ 
=\sum_{I^m \subset J^m_{l'}}  (-1)^{|I^m|} \cdot \Big(   \sum_{l''_J \leq (Z_K - l'(I^m)-E)_J \atop [l''] = [(Z_K - l'(I^m))]} z(l'')  \Big).
\end{align}

Now, if we consider the relative series $R(\mathbf{t}) = Z(\mathbf{t}) \cdot \prod_{ v\in J}(1-\mathbf{t}^{E_v^*})^{a_v}$, one can observe that the right hand side of (\ref{eq:finaleq1}) equals $\sum_{l'' \in S',[ l''] = [Z_K], l''_J \leq (Z_K-E)_J} r(l'')$, while the term appearing in the sum on the left is $h^1_{virt}(l'(I^m))$. This finishes the proof of  (\ref{dualitycounting2}).

\subsection{Another incarnation of Theorem \ref{topdualcount}}

In this section we present yet another form of the main formula (\ref{dualitycounting2}). 

The idea behind is to use instead of $R'(\bt)$ the function $R''(\mathbf{t}) = Z(\mathbf{t}) \cdot \prod_{ v\in J_{l'}}(1-\mathbf{t}^{-E_v^*})^{a_v}$, calculate the periodic constant of $R''_{0,J_{l'}}$ in two ways and compare the identic results. The calculations are similar to the case of  (\ref{dualitycounting2}) in the previous section, and their details are left to the reader. The new form which we get by performing the above described idea is as follows:

\begin{theorem}
Let's have an arbitrary resolution graph $\mathcal{T}$ which corresponds to a rational homology sphere link, then with the notations above we have:
\begin{align*}
\sum_{I^m \subset J^m}  (-1)^{|I^m|} \cdot \left( \chi(l'(I^m)) + \mathfrak{sw}_{[l'(I^m)]}^{norm} - \chi(r_{[l'(I^m)]}) - \sum_{l''_J  \leq (l'(I^m)-E)_J, [l''] = [ l'(I^m)]} z(l'')  \right) \\ 
\vspace{0.3cm}
= \sum_{l'' \in S',[ l''] = [Z_K], l''_J \ngeq (Z_K)_J} r(l'').
\end{align*}
\end{theorem}

\subsection{The special case $l' = E_v^*$ using lattice weights}

\subsubsection{} For an important  special case, when $l' = E_v^*$ for some $v \in \calv$, the main formula (\ref{dualitycounting2}) reads as follows:

\begin{equation}\label{eq:dualityspec}
h^1_{virt}(E_v^*)  =  \mathfrak{sw}^{norm}_0(\calt)  - \sum_{l'' \in S',[ l''] = [Z_K], l''_v \leq (Z_K-E)_v} r(l'').                                           
\end{equation}


 If we set the notation $h_v:=[E^*_v]$, then by the definition of $h^1_{virt}$, (\ref{eq:dualityspec}) is equivalent with the equation
\begin{equation*}
\mathfrak{sw}_0^{norm}(\calt)-  \chi(E_v^*) +\chi(r_{h_v})-\mathfrak{sw}_{h_v}^{norm}(\calt) + Q_{h_v}(E_v^*) =     \sum_{l'' \in S',[ l''] = [Z_K], l''_v \leq (Z_K-E)_v} r(l'').                                           
\end{equation*}
The relative Poincar\'e series associated with $E^*_v$ can be written as $R(\bt)=Z(\bt)-\bt^{E^*_v} Z(\bt)$, therefore we have 
 \begin{equation*}
 \sum_{l'' \in S',[ l''] = [Z_K], l''_v \leq (Z_K-E)_v} r(l'') =  \sum_{l'' \in S',[ l''] = [Z_K], l''_v < (Z_K)_v} z(l'') - \sum_{l'' \in S',[ l''] = [Z_K- E_v^*], l''_v < (Z_K-E_v^*)_v} z(l'').
\end{equation*}
Furthermore, the right hand side of the above identity is nothing else in this case than $Q_{[Z_K],v}(Z_K)-Q_{[Z_K]-h_v,v}(Z_K- E_v^*)$, which equals $\mathrm{pc}(Z_{0,v}(t_v))-Q_{[Z_K]-h_v,v}(Z_K- E_v^*)$ according to Theorem \ref{dualcount}. Notice also the fact that if $l'' \ngeq E_v^* $ and $l'' \in S'$ then we have $l''_v  < (E_v^*)_v$, hence $Q_{h_v}(E_v^*)=Q_{h_v,v}(E_v^*)$. Moreover, we have $\mathfrak{sw}_0^{norm}(\calt)=\mathrm{pc}(Z_{0,v}(t_v)) + \mathfrak{sw}_0^{norm}(\calt\setminus v)$ followed by the surgery formula from Theorem \ref{surgery}. 

Now, all these facts together imply that (\ref{eq:dualityspec}) is equivalent with

\begin{equation}\label{eq:equiveq}
\chi(E_v^*) -\chi(r_{h_v})+\mathfrak{sw}_{h_v}^{norm}(\calt) = \mathfrak{sw}_{0}^{norm}(\calt\setminus v)  +  Q_{[Z_K]-h_v,v}(Z_K- E_v^*) +  Q_{h_v,v}(E_v^*).                               
\end{equation}

\subsubsection{\bf A proof using the lattice weight function} In the following, we give the main ideas of the proof for equation (\ref{eq:equiveq}) using lattice weight function methods which play a key role in the definition of lattice cohomology, see \cite{Nlat}.
 For a comprehensive explanation we need to recall some terminology and results from \cite{NJEMS}. 
\vspace{0.2cm}

One can consider the cubical decomposition of $\mathbb{R}^{|\calv|}=L'\otimes \mathbb{R}$: for any $l'\in L'$ and $I \subset \calv$ arbitrary subset of vertices with $|I|=q$ we define the $q$-cube $(l', I)$ whose vertices are the cycles $\{l' + E_F | F \subset I\}$. Then the weight function on the set of all $q$-cubes is defined by 
$$w((l', I)) = \max_{F \subset I} \chi(l' + E_F).$$ 
In particular, it gives $w(l')=\chi(l')$ for any $l'\in L'$. 

The first result which we recall is the interpretation of the coefficients of the topological Poincar\'e series using the weights of the corresponding cubes, namely we have: 

\begin{theorem}{\cite[Theorem 2.3.10]{NJEMS}}
$z(l') = \sum_{I \subset \calv} (-1)^{|I| + 1} w((l', I))$.
\end{theorem}

For any $a \geq b, [a] = [b]\in H$ we define the rectangle $R(a, b)$ as the cubical complex containing all the cubes $(l', I)$, for which $a \leq l'$,  $l' + E_I \leq b$ and $[l']=[a]=[b]$. One can associate with such a rectangle the weighted Euler characteristic as follows
$$\chi_w(R(a,b))=\sum_{(l',I)\subset R(a,b)}(-1)^{|I| + 1} w((l', I)).$$
Then the following theorem expresses the normalized Seiberg--Witten invariant from this weighted Euler characteristic:

\begin{theorem}{\cite[Theorem 2.3.7]{NJEMS}}
If $h\in H$ an arbitrary class, $a \in L'$ with very low coordinates, $b \in L'$ which very large coordinates such that $[a]=[b] = h$, then we have the following:
\begin{equation}\label{eq:swchiw}
\mathfrak{sw}_h^{norm} =  \chi(r_h)-\chi_w(R(a, b)).
\end{equation}
\end{theorem}

Now we apply these theorems in our context as follows: let $a=\sum_{v}a_vE_v\in L'$ be a very small and $b=\sum_v b_v E_v\in L'$ a very large cycle such that $[a]= [b] = [E_{v}^*]=h_v$, and set the notation $d_v:=  (E_v^*)_v$ for simplicity. Then, using some easy computation which we omit from here, one gets the following identitites:

\begin{equation*}
\chi_w(R(a, b- b_v E_v + d_v E_v)) - \chi_w(R(a - a_v E_v + d_v E_v, b- b_v E_v + d_v E_v)) =  -  \sum_{l'_v  < d_v ,[ l'] = h_v}  z(l'). 
\end{equation*}

\begin{equation*}
\chi_w(R(a - a_v E_v + d_v E_v, b)) - \chi_w(R(a - a_v E_v + d_v E_v, b- b_v E_v + d_v E_v)) =   - \sum_{ l''_v < (Z_K-E_v^*)_v\atop [ l''] = [Z_K- E_v^*],} z(l''). 
\end{equation*}

\begin{equation*}
 \chi(R(a - a_v E_v + d_v E_v, b- b_v E_v + d_v E_v) =  -\mathfrak{sw}_{0}^{norm}(\calt\setminus v) + \chi(E_v^*).
\end{equation*}
Finally, using additivity of the weighted Euler characteristic, the equation (\ref{eq:equiveq}) follows from the above identities and (\ref{eq:swchiw}).

\section{Dominance property of natural line bundles on splice quotient singularities}\label{s:domspq}

\subsection{}In this section we investigate the dominance property of line bundles in the case of splice quotient singularities.

We will prove, using formulae presented in section \ref{s:motspq} and \ref{s:topdcf}, that in some cases a natural line bundle often has higher $h^1$ value than many line bundles in the image of the Abel map. However, this phenomena is not true in the highest generality and there might exist line bundles in the Picard group which have higher $h^1$ value than the natural line bundles. These unfavorable situations will be supported by examples given at the end of the section. 

First of all, on the positive side we prove the following result:

\begin{theorem}\label{thm:domsplice}
Let $(X,0)$ be a splice quotient singularity with rational homology sphere link and let $\calt$ be its dual resolution graph which satisfies the monomial conditions and $\tX$ the corresponding resolution space with exceptional divisor $E= \cup_{v \in \calv} E_v $. 

Let $l' = a_v E_v^*$ be a cycle for some vertex $v \in \calv$, such that $a_v > 0$ and $E_v^* \in S_{an}$.  \\
(a) \ If $a_v = 1$ one has $h^1(\tX, \calL) \leq h^1( \calO_{\tX}(-E_v^*))$ for every line bundle $\calL \in \im(c^{-E_v^*}(Z))$. \\
(b) \ If $a_v >1$ we also assume that the line bundle $\calO_{\tX}(- E_v^*)$ has no base point on the exceptional divisor $E_v$. In this case there are only finitely many points $p_1, \cdots, p_k \in E_v$ such that the line bundle $\calO_{\tX}(- E_v^*)$ has no section through $p_i$. Then, for any line bundle $\calL = \calO_{\tX}( \sum_{1 \leq k \leq a_v} D_k)$ given by disjoint transversal cuts $D_{ k},  1 \leq k \leq a_v$ at regular points of the exceptional
divisor $E_v$, such that none of the transversal cuts go through the critical points $ p_1, \cdots, p_k$, one has the following inequality:
$$h^1(\tX, \calL) \leq h^1(\tX, \calO(-l')).$$
\end{theorem}

\begin{remark}
Notice that, for $a_v>1$, the base point freeness assumption is  a combinatorial property computable from the resolution graph.
\end{remark}

%
%
%

\subsection{Proof of Theorem \ref{thm:domsplice}}

Let us denote the  intersection points $D_{k} \cap E_v$ by $q_{ k}$ for any $1 \leq k \leq a_v$. We  fix disjoint transversal cuts $N_{ k}, 1 \leq k \leq a_v$, such that $\calO_{\tX}(N_{ k}) = \calO_{\tX}(- E_v^*)$. Denote the intersection points by $w_{ k} = N_k \cap E_v$  and assume that they are distinct from the intersection points $q_{k}$ if $a_v >1$.

We fix a set of nonnegative integers $s = \{s_{ k} \ : \ 1 \leq k \leq a_v\}$ and blow up $\tX$ at the points $q_{k}$ sequentially along the $D_{k}$ divisors  $s_{k}$-many times. We denote the resulting space by $\tX_s$. 

Similarly, one can blow up $\tX$ at the points $w_{k}$ sequentially along the $N_{k}$ divisors  $s_{k}$-many times, and denote the result by $\tX_{s, n}$. 

Notice that the analytical type $\tX_{s, n}$ is splice quotient, because it certainly satisfies the end-curve condition, cf. \ref{ss:spq}.
Moreover, the resolutions $\tX_{s, n}$ and $\tX_{s}$ have the same resolution graph, which will be denoted by $\mathcal{T}_s$, and its set of vertices by $\calv_s$. 
For any $k$, let us set the notation  $v_{s, k}$ for the end-vertex of $\mathcal{T}_s$ appearing on the newly created $k$-th leg connected to $v$. If $s_{k} = 0$ then we set $v_{s, k} = v$.

We write  $D_{s, k}$ for the strict transform of the divisor $D_{k}$ on $\tX_s$, and $N_{s, k}$ for the strict transform of the divisor $N_{k}$ on $\tX_{s, n}$. Then we define the cycle $l'_s = \sum_{1 \leq k \leq a_v} E_{v_{s, k}}^*$  with support  $I_s:=|l'_s|^*$, and the line bundle  $\calL_s = \calO_{\tX_{s}}(\sum_{ 1 \leq k \leq a_v }D_{s, k})$. Notice that, since $\calO_{\tX_{s, n}}(N_{s, k}) = \calO_{\tX_{s, n}}(- E_{v_{s, k}}^*) $, we have $\calO_{\tX_{s, n}}(\sum_{ 1 \leq k \leq a_v }N_{s,  k}) = \calO_{\tX_{s, n}}(- l'_s)$.

We say that a set of nonnegative integers $s = \{s_{ k} \ : \ 1 \leq k \leq a_v\}$ is \emph{admissible} if for every $1 \leq k_1, k_2 \leq a_{v}$ one has $|s_{ k_1} - s_{k_2}| \leq 1$. For such an admissible $s$ we define $|s| = \sum_{1 \leq k \leq a_v} s_{k}$. Then we will prove the following:
\vspace{0.1cm}

{\bf Claim:} \ \emph{For an admissible $s$ one has $h^1(\tX_s, \calL_s) \leq h^1(\calO_{\tX_{s, n}}(- l'_s))$.}
\vspace{0.1cm}

The proof will be given by downgoing induction on $|s|$.

First of all, if $|s|$ is very large, then $h^1(\tX_s, \calL_s) = h^1(\calO_{\tX_{s, n}}(- l'_s)) = p_g(X,o)$ by Proposition \ref{0dim} and the statement is trivial. 

Now, consider an arbitrary admissible $s$ and assume that $\min_{1 \leq k \leq a_v} s_{k} = d$. We choose an index $1 \leq t \leq a_v$  for which this minimum is realized, ie.  $s_{t} = d$. Let us define the set of integers $s^{t}=\{s^{t}_k \ : \ 1 \leq k \leq a_v\}$ such that $s^{t}_{k} = s_{k} + 1$ if $k = t$ and $s^{ t}_{k} = s_{k}$ otherwise.  Since $s_{t} = d$ it follows that $s^t$ is also admissible. Moreover, $|s^{ t}| = |s| + 1$, so by the induction hypothesis we know that 
$$h^1(\tX_{s^{t}}, \calL_{s^{ t}}) \leq h^1(\calO_{\tX_{s^{t}, n}}(- l'_{s^{t}})).$$

Furthermore, notice that $h^1(\calO_{\tX_{s^{t}, n}}(- l'_{s^{t}})) = h^1(\calO_{\tX_{s, n}}(- l'_s)) + b_n$, where $b_n= 0$ if the line bundle $\calO_{\tX_{s, n}}(- l'_s)$ has no base point at the intersection point $N_{s, t} \cap E_{v_{s, t}}$, or $b_n= 1$ if  $\calO_{\tX_{s, n}}(- l'_s)$ has  a base point at  $N_{s, t} \cap E_{v_{s, t}}$.

 Similarly we have $h^1(\tX_{s^{ t}}, \calL_{s^{t}}) = h^1(\tX_s, \calL_s) + b$, where $b= 0$ if the line bundle $ \calL_s$ has no base point at $D_{s, t} \cap E_{v_{s, t}}$, or $b= 1$ if $ \calL_s$ has  a base point at $D_{s, t} \cap E_{v_{s, t}}$.
By using the induction hypothesis the case $b_n=0$ follows automatically, hence it is enough to prove the inequality in the case when $b_n = 1$.

Therefore, in the sequel we assume that $h^1(\calO_{\tX_{s^{t}, n}}(- l'_{s^{t}})) = h^1(\calO_{\tX_{s, n}}(- l'_s)) + 1$.

We denote the set of arrows in $l'_s$ by $J_s$. For a subset $A \subset J_s$ and a vertex $w \in \calv_s$ we denote by $c_w(A)$ the number of arrows in $A$ supported on $E_w$, and define the cycle $l'_A = \sum_{w \in \calv_s} - c_w(A) E_w^*$.

We also introduce  the support $J(A) = |l'_A|^*$ of a subset $A \subset J_s$.  

Consider the relative topological Poincar\'e series associated with $\calt_s$ and the cycle $-l'_A$:

\begin{equation*}
R_{s, A}(\mathbf{t}) = \prod_{w \in \calv_s}(1-\mathbf{t}^{E_w^*})^{\delta_w -2} \cdot \prod_{w \in \calv_s}(1-\mathbf{t}^{E_w^*})^{c_w(A)}=\sum_{l''} r_{s,A}(l'')\bt^{l''}.
\end{equation*}

Then, by Theorem \ref{dualityspq} we have the following identity

\begin{equation}\label{formula1}
\sum_{B \subset A} (-1)^{|A|} h^1(\calO_{\tX_s}( - l'_B)) = \sum_{l'' \in S',[ l''] = [(Z_K)_s], l''_{K(A)} \leq ((Z_K)_s-E)_{J(A)}} r_{s, A}(l'').
\end{equation}

Let us rewrite this identity as follows. Let $M$ be the set of the intersection points  $q_{ k}$ for $1 \leq k \leq a_v$, and for a subset $A \subset M$ fix the notation $s_{A} = \max_{q_{k} \in A} s_{k}$. Then, we consider the multivariable series
\begin{equation*}
R^*_{s, A}(\mathbf{t}) = \prod_{w \in \calv}(1-\mathbf{t}^{E_w^*})^{\delta_w -2} \cdot (1-\mathbf{t}^{E_v^*})^{|A|}=\sum_{l''} r^*_{s,A}(l'')\bt^{l''},
\end{equation*}
where $|A|$ is the cardinality of $A \subset M$. 

Obviously we can identify the set of arrows $J_s$ with $M$ which induces the following equality:

\begin{equation*}
\sum_{B \subset A} (-1)^{|A|} h^1(\calO_{\tX_s}( - l'_B)) = \sum_{l'' \in S',[ l''] = [Z_K], l''_v \leq (Z_K -E)_v - s_{A} }  r^*_{s, A}(l'').
\end{equation*}
Indeed, we get this formula if we blow up the the points $q_k \in A$ $s_{ k}$ many times and apply (\ref{formula1}) in that situation. 

From this one can deduce the equality
\begin{equation*}
 p_g - h^1(\calO_{\tX_s}(- l'_s)) =  \sum_{\emptyset \neq A \subset J_s} (-1)^{|A| + 1}\left( \sum_{l'' \in S',[ l''] = [Z_K], l''_v \leq (Z_K -E)_v - s_{A} }  r^*_{s, A}(l'') \right).
\end{equation*}

Similarly, one gets the identity for $s^{ t}$ as well, namely

\begin{equation*}
 p_g - h^1(\calO_{\tX_{s^{t}}}( - l'_{s^{t}})) =  \sum_{\emptyset \neq A \subset J_{s^{t}}} (-1)^{|A| + 1}\left(   \sum_{l'' \in S',[ l''] = [Z_K], l''_v \leq (Z_K -E)_v - s^t_{A} }  r^*_{s^t, A}(l'')                                   \right).
\end{equation*}

We can identify $J_s$ with $J_{s^{t}}$ by identifying the $t$-th cut with its blow up.
Notice also that $R^*_{s, A}(\mathbf{t}) = R^*_{s^t, A}(\mathbf{t})$ and $h^1(\calO_{\tX_{s^{t}, n}}(- l'_{s^{t}})) = h^1(\calO_{\tX_{s, n}}(- l'_s)) + 1$, from which we get    

\begin{equation*}
 1 =  \sum_{\emptyset\neq \in A \subset J_{s}} (-1)^{|A| + 1}\left( \sum_{l'' \in S', [ l''] = [Z_K], (Z_K -E)_v - s^t_{A} < l''_v \leq (Z_K -E)_v - s_{A} }  r^*_{s^{t}, A}(l'') \right).
\end{equation*}
One can see that the conditions $ (Z_K -E)_v - s^t_{A} < l''_v \leq (Z_K -E)_v - s_{A}$ can only hold if $s_t = s_{A}$ and in this case $ l''_{v} =  (Z_K -E)_v - s_{ t} = (Z_K -E)_v - s_{A}$.

In the following let us denote  the subset $M' = \{q_k \ : \ 1 \leq k \leq a_v, s_k = d\} \subset M$, where we have $q_t \in M'$.

From these we get that there is an $l'' \in S'$, $[ l''] = [Z_K]$, such that $l''_{v} =  (Z_K -E)_u - d$ and  

\begin{equation*}
 \sum_{q_{t} \in A \subset M'}  (-1)^{|A| + 1} \cdot  r^*_{s, A}(l'') \neq 0.
\end{equation*}

Let us consider $l'' = \sum_{u \in \calv} c_u E_u^*$. We know that if $c_u \neq 0$, then $u$ is a node or an end-vertex, or $u = v$, otherwise $ r^*_{s, A}(l'')  = 0$ for all subsets $q_{t} \in A \subset M'$. 

Then, first we claim  that $-(l'', E_v) = c_v \geq |M'| -1$. Indeed, let us denote $|M'| = r \geq 1$. Then using the expression of the Taylor expansion we get the following identitites:

\begin{equation*}
 \sum_{q_{ t} \in A \subset M'}   (-1)^{|A| + 1} \cdot r^*_{s, A}(l'') =  \sum_{q_{ t} \in A \subset M'}  (-1)^{|A| + 1} \cdot \prod_{u \neq v} \binom{\delta_u - 2}{c_u} \cdot \binom{\delta_v - 2 + |A|}{c_v};
\end{equation*}

\begin{equation*}
0 \neq \sum_{q_{ t} \in A \subset M'}  (-1)^{|A| + 1} \cdot  r_{s, A}(l'') =   \prod_{u \neq v} \binom{\delta_u - 2}{c_u} \cdot \sum_{0 \leq i \leq r-1}(-1)^i \cdot \binom{r-1}{i} \cdot \binom{\delta_v - 1 + i }{c_v};
\end{equation*}

\begin{equation*}
0 \neq \sum_{0 \leq i \leq r-1}(-1)^i \cdot \binom{r-1}{i} \cdot \binom{\delta_v - 1 + i }{c_v}.
\end{equation*}

It is well known that if $p(x)$ is a polinomial of degree at most $n \geq 0$ and $x$ is an arbitrary real number, then $\sum_{0 \leq i \leq n+1}(-1)^i \cdot \binom{n+1}{i} \cdot p(x + i) = 0$, which in our case yields that $c_v \geq r -1$.

In summary,  we have proved that there is a cycle $l'' = \sum_{u \in \calv} c_u E_u^* \in S'$ such that $[l''] = [Z_K]$, $l''_{v} =  (Z_K -E)_u - d$ and $|l''|^* \subset \mathcal{E}  \cup \caln \cup v$ with 
 $c_v \geq |M'| -1$.
\vspace{0.2cm}
 
\underline{Now, assume first that $a_v >1$.}
\vspace{0.2cm}

Since  every $u \neq v, u \in |l''|^*$ is an end-vertex or a node, let us fix a cut $C_u$, such that $\calO_{\tX}(C_u) = \calO_{\tX}(-E_u^*)$. If $c_v \neq 0$ (which is the case if $|M'| >1$) then we also fix cuts $C_1, C_2, \cdots, C_{r}$ such that $\calO_{\tX}(C_i) = \calO_{\tX}(-E_v^*)$ and $C_{1}, \cdots, C_{r-1}$ go through points from $M' \setminus q_t$. Note that the existence of these cuts is guaranteed by the assumptions of the theorem, since $M' \setminus q_t$ is disjoint from $p_1, \cdots, p_k$.

Furthermore, $\calO_{\tX}(-l'') = \calO_{\tX}(\sum_{u \neq v}c_u \cdot C_u + (c_v - r+1) \cdot C_r + \sum_{1 \leq i \leq r-1} C_i)$, which means that there is a section in $H^0(\calO_{\tX}(K + (Z_K - l'')))_{reg}$ vanishing at the points in $M' \setminus q_t$.

Therefore, there is a differential form $\omega$ which has a pole on the exceptional divisor $E_v$ of order $d+ 1$, and  $\omega$ has arrows at the points from $M' \setminus q_t$, while it has no arrow at the point $q_t$. This means, that $\omega$ gives a differential form on $\tX_s$ which has a pole along the divisor $\sum_{ 1 \leq k \leq a_v }D_{s, k}$, but it has no pole along the divisor $\sum_{ 1 \leq k \leq a_v }D_{s^t, k}$. This yields that $h^1(\tX_{s^{ t}}, \calL_{s^{t}}) = h^1(\tX_s, \calL_s) + 1$, so $b=1$ and the result follows in this case.
\vspace{0.2cm}

\underline{Assume in the following that $a_v = 1$.}
\vspace{0.2cm}

Notice that when the points $q_1, w_1$ are different, then the very same proof works, since at the last step we can fix a cut $C_1$ such that $\calO_{\tX}(C_1) = \calO_{\tX}(-E_v^*)$ and        
$C_1$ does not go through the point $q_1$. 
In this case again one constructs a differential form $\omega$ with a pole on the exceptional divisor $E_v$ of order $d+ 1$ and with no arrow at the point $q_1$,  which finishes the proof exactly in the same way.

On the other hand, if the points $q_1, w_1$ can not be choosen to be different, then the line bundles $\calO_{\tX}(N_1), \calO_{\tX}(D_1)$ have common base point at $q_1$. 
However, if we blow up these two base points until one of them stops having a base point, then the same proof can be repeated. This finishes the proof also in the case $a_v = 1$.

\subsection{Counterexample for the general case}\label{ss:cex}

In this section we illustrate by an example that the assumption in Theorem \ref{thm:domsplice} is essential. Thus, in general there might exists  a splice quotient singularity $(X,0)$ with its resolution $\tX$ together with a cycle $l' \in \calS'$ and a line bundle $\calL \in \pic^{-l'}(\tX)$, such that $h^1(\tX, \calO_{\tX}(-l')) < h^1(\tX, \calL)$.

First of all, consider the star-shaped rational homology sphere graph $\mathcal{T}_1$ with vertex set  $\calv_1$ as shown by the left side of the following picture. 
\begin{figure}[h!] \label{fig:tt1}
\begin{center}
\begin{tikzpicture}[scale=.75]
\node (v7) at (8.5,0) {};
\node (v5) at (7,0) {};
\node (v11) at (5.5,-1) {};
\node (v10) at (5.5,1) {};
\node (v12) at (5.5,0.3) {};
\node (v13) at (5.5,-0.3) {};
\draw[fill] (8.5,0) circle (0.1);
\draw[fill] (7,0) circle (0.1);
\draw[fill] (5.5,-1) circle (0.1);
\draw[fill] (5.5,1) circle (0.1);
\draw[fill] (5.5,0.3) circle (0.1);
\draw[fill] (5.5,-0.3) circle (0.1);
\draw  (v7) edge (v5);
\draw  (v5) edge (v10);
\draw  (v5) edge (v11);
\draw  (v5) edge (v12);
\draw  (v5) edge (v13);
\node at (7,0.5) {\small $-3$};
\node at (8.5,0.5) {\small $-1$};
\node at (4.8,1) {\small $-N$};
\node at (4.8,0.3) {\small $-N$};
\node at (4.8,-0.3) {\small $-N$};
\node at (4.8,-1) {\small $-N$};
\node at (5.5,1.3) {\tiny $v_1$};
\node at (5.5,0.6) {\tiny $v_{2}$};
\node at (5.5,-0.6) {\tiny $v_{3}$};
\node at (5.5,-1.3) {\tiny $v_{4}$};
\node at (8.9,0) {\tiny $v_{5}$};
\node at (7,-0.4) {\tiny $n$};
\node at (7,-1.7) {$\calt_1$};
\end{tikzpicture}
\hspace{2.5cm}
\begin{tikzpicture}[scale=.75]
\node (v7) at (8.5,0) {};
\node (v5) at (7,0) {};
\node (v11) at (5.5,-1) {};
\node (v10) at (5.5,1) {};
\node (v12) at (5.5,0.3) {};
\node (v13) at (5.5,-0.3) {};
\node (v14) at (10,0) {};
\node (v15) at (11.5,1) {};
\node (v16) at (11.5,0.3) {};
\node (v17) at (11.5,-0.3) {};
\node (v18) at (11.5,-1) {};
\draw[fill] (8.5,0) circle (0.1);
\draw[fill] (7,0) circle (0.1);
\draw[fill] (5.5,-1) circle (0.1);
\draw[fill] (5.5,1) circle (0.1);
\draw[fill] (5.5,0.3) circle (0.1);
\draw[fill] (5.5,-0.3) circle (0.1);
\draw[fill] (10,0) circle (0.1);
\draw[fill] (11.5,-1) circle (0.1);
\draw[fill] (11.5,1) circle (0.1);
\draw[fill] (11.5,0.3) circle (0.1);
\draw[fill] (11.5,-0.3) circle (0.1);
\draw  (v7) edge (v5);
\draw  (v5) edge (v10);
\draw  (v5) edge (v11);
\draw  (v5) edge (v12);
\draw  (v5) edge (v13);
\draw  (v7) edge (v14);
\draw  (v14) edge (v15);
\draw  (v14) edge (v16);
\draw  (v14) edge (v17);
\draw  (v14) edge (v18);
\node at (7,0.5) {\small $-3$};
\node at (8.5,0.5) {\small $-M$};
\node at (4.8,1) {\small $-N$};
\node at (4.8,0.3) {\small $-N$};
\node at (4.8,-0.3) {\small $-N$};
\node at (4.8,-1) {\small $-N$};
\node at (10,0.5) {\small $-3$};
\node at (12.2,1) {\small $-N$};
\node at (12.2,0.3) {\small $-N$};
\node at (12.2,-0.3) {\small $-N$};
\node at (12.2,-1) {\small $-N$};
\node at (5.5,1.3) {\tiny $v_{11}$};
\node at (5.5,0.6) {\tiny $v_{12}$};
\node at (5.5,-0.6) {\tiny $v_{13}$};
\node at (5.5,-1.3) {\tiny $v_{14}$};
\node at (11.5,1.3) {\tiny $v_{21}$};
\node at (11.5,0.6) {\tiny $v_{22}$};
\node at (11.5,-0.6) {\tiny $v_{23}$};
\node at (11.5,-1.3) {\tiny $v_{24}$};
\node at (8.5,-0.4) {\tiny $v$};
\node at (7,-0.4) {\tiny $n_1$};
\node at (10,-0.4) {\tiny $n_2$};
\node at (8.5,-1.7) {$\calt$};
\end{tikzpicture}
\end{center}
\caption{The graphs $\calt_1$ and $\calt$}
\end{figure}

$\calt_1$ admits a splice quotient analytic structure (weighted homogenous in fact). We consider its topological Poincar\'e series $ Z^{\mathcal{T}_1}(\mathbf{t}) = \prod_{v \in \calv_1} (1-\mathbf{t}^{E_v^*})^{\delta_v -2}=\sum_{l''}z^{\calt_1}(l'')\bt^{l''}$.
 Then one can check that 
\begin{equation}\label{eq:e1}
\sum_{[ l''] = 0, l'' \geq 0, l''_{v_5} = l''_{n} = 0} z^{\mathcal{T}_1}(l'') = 1,
\end{equation}
since the only element $l'' \in \calS'(\calt_1)$ such that $ l''_{v_5} = l''_{n} = 0$ is $0$. 

Now, if we consider the cycle $l'_2 = 3 E_n^* - E_{v_1}^* - E_{v_2}^* - E_{v_3}^* - E_{v_4}^*$ then we claim that

\begin{equation}\label{eq:e2}
\sum_{[ l''] =[ l'_2] , l'' \geq l'_2, l''_{v_5} = (l'_2)_{v_5},  l''_{n} = (l'_2)_{n}} z^{\mathcal{T}_1}(l'') = 0.
\end{equation}

Indeed if there is an $l'' \in \calS'(\calt_1)$ such that $[ l''] =[ l'_2] , l'' \geq l'_2$ then we must have $(l'')_{v_i} > (l'_2)_{v_i}, 1 \leq i \leq 4$. However, we know that $l''_{v_5} = (l'_2)_{v_5},  l''_{n} = (l'_2)_{n}$
which yields that $(l'' , E_n) > 0$, which is a contradiction.

As a second step, we construct the graph $\mathcal{T}$ by gluing together two  pieces of  $\mathcal{T}_1$ along their $v_5$ vertices. This common vertex will be denoted by $v$ and has self-intersection number $-M$, where $M$ is a very large integer number so that we get a negative definite graph, see the right side of the above picture.

One can check that $M$ can be chosen in a way that $\mathcal{T}$ satisfies the monomial conditions too. Indeed, there are only two nodes  $n_1$ and $n_2$,  and they are in symmetric situation. So it is enough to see that the monomial conditions holds for the node $n_1$. For the branches containing $v_{1, i}$ for $1 \leq i \leq 4$ the monomial conditions trivially holds and for the branch containing the node $n_2$ it holds if  $-M$ is very low compared to the other self-intersection numbers.

Let $\tX$ be the resolution space with resolution graph $\mathcal{T}$ of the splice quotient singularity with rational homology sphere link $M$,  and consider the cycle $l'_c = E_{v}^*$. 

Then we claim the following proposition:
\begin{proposition}
 If $M$ is enough large, then for a generic line bundle $\calL \in \im(c^{-l'_c}(Z))$ one has $h^1(\tX, \calL) > h^1(\tX, \calO_{\tX}(-l'_c))$.
\end{proposition}

\begin{proof}
Notice first that, using the calculation of coefficients of the dual cycles from eg. \cite[7.1]{LN}, as the value of $M$ is increased, all the coeficcients of $ E_v^*$ decrease. This means that $M$ can be chosen such a way that $E_v^* = r_h \in L'$, where $h = [E_v^*]$.

We show first the inequality $h^1(\tX, \calL) \geq p_g - 1$.  
Indeed, we know from Proposition \ref{redcycle} that $\dim( \im(c^{-l'_c}(Z_{coh}))) = \dim( \im(c^{-l'_c}(Z)))$, where $Z_{coh}$ is the cohomological cycle and $Z$ is large. On the other hand $(Z_{coh})_v = 1$ since  $-M$ is very low, so we have $\dim( \im(c^{-l'_c}(Z_{coh}))) \leq -(Z_{coh}, E_v^*) = 1$. This implies $h^1(\tX, \calL) \geq p_g - 1$ by Theorem \ref{dimgen}. 

Secondly, we will prove in the following that
\begin{equation}\label{eq:claim2}
 h^1(\tX, \calO_{\tX}(-l'_c)) = p_g - 2.
\end{equation} 
The next computation will also show the essence of this counterexample.

First, we observe that by the above discussion one has $h^1(\tX, \calO_{\tX}(-l'_c))  = h^1(\tX, \calO_{\tX}(-r_h)) = \mathfrak{sw}_{h}^{norm}(\calt)$. 

We blow up $\mathcal{T}$ at the vertex $v$ and let us denote the new resolution graph by $\mathcal{T}_{new}$ with the newly created vertex $v_{new}$ and exceptional divisor $E_{new}$. We will use notations $L(\calt_{new}), L'(\calt_{new})$ for the corresponding objects associated with $\calt_{new}$.

Note that we also have $r_h = E_v^*\in L'(\calt_{new})$, where $h=[E_v^*]\in L'(\calt_{new})/L(\calt_{new})$ and one has 
\begin{equation}\label{eq:blowup}
E^*_v + E_{new} = E_{new}^*.
\end{equation}
Consider the following counting function defined in (\ref{eq:countintro}):

 \begin{equation*}\label{eq:countintro2}  
Q^{\calt_{new}}_{h,new}: L'_{h}(\calt_{new})\to \bZ, \ \ \ \ Q^{\calt_{new}}_{h,new}(l'_0):=\sum_{l'_{new} \ngeq l'_{0,new},   
\, [l']=[l'_0]} z^{\calt_{new}}(l'), 
\end{equation*}
where $l'_{new}:=l'|_{E_{new}}$. 
Note that $Q^{\calt_{new}}_{h,new}$ is the counting function of the reduced Poincar\'e series $Z^{\calt_{new}}_{h,new}(t_{new})$, where $t_{new}$ is set for the variable associated with the vertex $v_{new}$.

Let us denote the subgraph of $\calt_{new}$ with vertex set $\calv_{new} \setminus v_{new}$ by $\mathcal{T}'$, and let $M'$ be the negative definite plumbed 3-manifold associated with $\calt'$.

Assume that $l'_0=\sum_{u\in\calv} a_uE^*_u\in L'$ and all $a_u$ are sufficiently large. Then if we write $l'_0=r_h+l$ where $l\in L_{new}$ and $h=[l'_0]$, by Theorem \ref{th:NJEMSThm} $Q^{\calt_{new}}_{h,new}(l'_0)$ equals
 a quasipolynomial  $\mathfrak{Q}^{\calt_{new}}_{h,new}(l)$  defined on $L_{new}$. Moreover, the surgery formula Theorem \ref{surgery} can also be applied and it gives 
 
\begin{eqnarray*}\label{eq:QPrestr}
\mathfrak{Q}^{\calt_{new}}_{h,new}(l):=-
\mathfrak{sw}_{-h*\sigma_{can}}(M)-\frac{(-Z_K+2r_h+2l)^2+|\mathcal{V}|}{8}
+ \mathfrak{sw}_{-[j^*(r_h+l)]*\sigma_{can'}}(M') \\ +\frac{(-Z_K^{\calt'} +
 2j^*(r_h+l))^2+|\mathcal{V}(\calt')|}{8}.
 \end{eqnarray*} 

We apply this surgery formula for the case $l'_0=r_h+l$, where $r_h=E^*_v$ as above and $l = E_{new} + t \cdot x$, for some  $t$ very large and divisible by $|H|$. Then taking the constant terms in $x$ on both sides we get 

\begin{equation*}
 \mathfrak{Q}^{\calt_{new}}_{h,new}(E_{new}) = \mathfrak{sw}^{norm}_h(\calt) + \chi(E_{new}^*) - \chi(E_v^*) -\mathfrak{sw}^{norm}_{0}(\calt'),
\end{equation*}
hence by (\ref{eq:blowup}) one gets
\begin{equation}
\mathfrak{Q}^{\calt_{new}}_{h,new}(E_{new}) = \mathfrak{sw}^{norm}_h(\calt)  -\mathfrak{sw}^{norm}_{0}(\calt')+1.
\end{equation}

Finally, we claim that $p_g = \mathfrak{sw}^{norm}_{0}(\calt')$ and $\mathfrak{Q}^{\calt_{new}}_{h,new}(E_{new}) = -1$, which would imply the desired equation (\ref{eq:claim2}).

For the first, we use the following identity given by Theorem \ref{surgery}:
\begin{equation*}
p_g = \mathfrak{sw}^{norm}_{0}(M) = \mathfrak{sw}^{norm}_{0}(M') +  \textnormal{pc}(Z^{\calt_{new}}_{0}(t_{new})).
\end{equation*}
So it is enough to prove $\textnormal{pc}(Z^{\calt_{new}}_{0}(t_{new}))=0$.

By Theorem \ref{dualcount} we can express the periodic constant of the $(h=0)$-part of the topological Poincar\'e series as a finite sum of coefficients of the $(h=Z_K)$-part of the topological Poincar\'e series in the following way

\begin{equation*}
\textnormal{pc}(Z^{\calt_{new}}_{0,new}(t_{new})) = \sum_{[l'] = [Z_K^{\mathcal{T}_{new}}], l'_{new} < (Z_K^{\mathcal{T}_{new}})_{new}} z^{\mathcal{T}_{new}}(l').
\end{equation*}
Then  $\textnormal{pc}(Z^{\calt_{new}}_{0,new}(t_{new})) =0$ follows, since $(Z_K^{\mathcal{T}_{new}})_{new} = (Z_K - E)_v$  is very small if the self-intersection $-M$ is low enough.

Now we prove the equality $\mathfrak{Q}^{\calt_{new}}_{h,new}(E_{new}) = -1$.

By Theorem \ref{depth} we know that if $l' = r_h + l \in \sum_{v \in \calv} (\delta_v - 2) E_v^* + int(S')$, then we have

\begin{equation*}
\mathfrak{Q}^{\calt_{new}}_{h, new}(l) = Q^{\calt_{new}}_{h, new}(l') = \sum_{[l''] = [l'], l''_{new} \ngeq l'_{new}} z^{\mathcal{T}_{new}}(l'').
\end{equation*}

Consider the cycle $l' = E_{new}^* + E - E_{new} + 3 E_{n_1} + 3 E_{n_2}$. 
If the selfintersection numbers $-N, -M$ are low enough, then $l'$ definitely satisfies the condition of Theorem \ref{depth}, hence we get

\begin{equation}
\mathfrak{Q}^{\calt_{new}}_{h,new}(E_{new}) = Q_{h, new}(l') = \sum_{l \geq 0, l_{new} = 0} z^{\mathcal{T}_{new}}(E_v^* + l).
\end{equation}
Thus,  we have to prove that $\sum_{l \geq 0, l_{new} = 0} z^{\mathcal{T}_{new}}(E_v^* + l) =  \sum_{l \geq 0, l_{new} = 0, l_v = 0} z^{\mathcal{T}_{new}}(E_v^* + l) = -1$.  

Indeed, one can see the decomposition (for a more precise explanation of this decomposition see section 8):

\begin{equation}
\sum_{l \geq 0, l_{new} = 0, l_v = 0} z^{\mathcal{T}_{new}}(E_v^* + l) = C_{1, 0} \cdot C_{2, 1} + C_{1, 1} \cdot C_{2, 0} - C_{1, 0} \cdot C_{2, 0},
\end{equation}
where we have $$C_{i, 0} = \sum_{[ l''] = 0, l'' \geq 0, l''_{v_{i, 5}} = l''_{i, n} = 0} z^{\mathcal{T}_{i}}(l'') = 1$$ and $$C_{i, 1} = \sum_{[ l''] =[ l'_{i , 2}] , l'' \geq l'_{i, 2}, l''_{v_{i, 5}} = (l'_{i, 2})_{v_{i, 5}},  l''_{n_i} = (l'_2)_{n_i}} z^{\mathcal{T}_{i}}(l'')  = 0,$$ implied by the equations 
(\ref{eq:e1}) and (\ref{eq:e2}).
\end{proof}

\section{Dominance property of natural line bundles on weighted homogenous singularities}\label{s:domwh}

In this section we show that, unlike the more general case of splice quotient singularities,  if we have a weighted homogenous singularity $(X, 0)$ with resolution $\tX$ and a cycle $l' \in S'$, then one has the inequality $$h^1(\tX, \calL) \leq h^1(\calO_{\tX}(-l'))$$ for any line bundle $\calL \in \pic^{-l'}(\tX)$. 

In the following, first we give a brief introduction to the resolution of weighted homogenous singularities.

\subsection{Preliminaries}\label{ss:WHprel} Let $(X,0)$ be a normal weighted homogeneous surface singularity. It is the germ at the origin of an affine variety $X$ with good $\mathbb{C}^*$-action, which means that its affine coordinate ring is $\mathbb{Z}_{\geq 0}$--graded. The dual graph $\calt$ of the  minimal good resolution is star--shaped, and the $\mathbb{C}^*$-action
of the singularity induces an $S^1$--Seifert action on the link of the singularity. In particular, 
the link of $(X,0)$ is a negative definite Seifert 3--manifold characterized by its
normalized Seifert invariants $Sf=(-b_0,g;(\alpha_j,\omega_j)_{j=1}^\nu)$, see eg. \cite{neumann}. Furthermore, similarly as before we assume that the link is  a rational homology sphere, hence $g=0$.

More precisely, if we write $v_0$ for the central vertex, then $\Gamma\setminus v_0$ consists of $\nu$ legs, each leg consisting of $s_j$ vertices. Here we assume that $\nu\geq 3$. Then $-b_0$ is the  Euler number of $v_0$, and the Euler numbers of the vertices  $v_{ji}$ of the $j^{th}$ leg $(1\leq j\leq \nu)$, denoted by 
$-b_{j1},\ldots, -b_{js_j}$ with $b_{ji}\geq 2$, can be 
determined by the Hirzebruch–Jung negative continued fraction
$\alpha_j/\omega_j=[b_{j1}, \dots, b_{js_j}]$,  where
$\gcd(\alpha_j,\omega_j)=1, \ 0<\omega_j<\alpha_j$. On each leg $v_{j1}$ is connected to $v_0$ by an edge. For any $j$, we also introduce $0<\omega_j'<\alpha_j$ such that
$\omega_j\omega_j'-1= \alpha_j\tau_j$ for some $\tau_j$.

We denote by $E_0$ and $E_{ji}$ the irreducible exceptional curves indexed by vertices $v_0$ and $v_{ji}$ . Let $P_j$ ($1\leq j\leq \nu$) be
$E_{v_0}\cap  E_{j1}$. Then it is known the following classification result, cf. \cite{Dolg,Pink}.
\begin{theorem}[\bf Analytic Classification
Theorem]\label{th:ACTh} 
The analytic isomorphism type of a normal surface weighted
homogeneous singularity (with rational homology sphere link)
with fixed Seifert invariants
is determined by the analytic type of
$(E_0,\{P_j\}_j)$
 modulo an action of ${\rm Aut}(E_0,\{P_j\}_j)$.
 (This is the same as the analytic classification of  Seifert
 line bundles over the projective line.)
 \end{theorem}

Next, we recall from \cite{NNA1} how to construct for any weighted homogeneous singularity $(X,0)$ the minimal good resolution
$\tX$ by an `analytic plumbing'.

Corresponding to the legs we fix distinct complex numbers $p_j\in \C$ as the affine coordinates of the
points $P_j$. Each leg,  with divisors $\{E_{ji}\}_{i=1}^{s_j}$,
$1\leq j\leq \nu$,
will be covered by open sets $\{U_{j,i}\}_{i=0}^{s_j}$, copies of $\C^2$
with coordinates $(u_{j,i},v_{j,i})$.
For each $1\leq i\leq s_j$ we glue $U_{j,i-1}\setminus \{u_{j,i-1}=0\}$ with
$U_{j,i}\setminus \{v_{j,i}=0\}$.
The gluing maps are
$v_{j,i}=u^{-1}_{j,i-1}$ $(1\leq i\leq s_j)$ and
$u_{j,i}$ equals $u^{b_{ji}}_{j,i-1}v_{j,i-1}$ for $2\leq i\leq s_j$ and
$u^{b_{j1}}_{j,0}(v_{j,0}-p_j)$ for $i=1$.

Furthermore, all $U_{j,0}$ charts will be identified to each other:
$u_{j,0}=u_{k,0},\ v_{j,0}=v_{k,0}$; denoted simply by $U_0$, with coordinates
$(u_0,v_0)$. Till now, the curve $E_{v_0}$ appears only in $U_0$, it has equation $u_0=0$.
In order to cover $E_{v_0}$ completely we need another copy $U_{-1}$ of $\C^2$ with coordinates
$(u_{-1},v_{-1})$ as well; the gluing of $U_0\setminus \{v_0=0\}$ with $U_{-1}
\setminus \{u_{-1}=0\}$ is  $v_0=u_{-1}^{-1}$, $u_0=u_{-1}^{b_0}v_{-1}$.

The curve $E_{j, i}$ appears on the $U_{j, i}$ chart as $u_{j, i} = 0$ and on the $U_{j, i-1}$ chart as $v_{j, i-1} = 0$ if $i >1$ and $v_{j, 0} = p_j$ if $i = 1$.

Let us denote the output space by $\widetilde{X^a}$. If we contract (analytically) 
$E=E_{v_0}\cup (\cup_{j,i}E_{ji})$ we get a germ of a
normal surface singularity $(X,0)$. In this context, $\tX$
(as a subset of $\widetilde{X^a}$) is  the pullback of a small Stein neighbourhood of $0$.
The following statement is proved in \cite{NBOOK}, basically it
follows from Theorem \ref{th:ACTh} and
from the fact that if we blow down the legs the obtained space carries naturally a Seifert
line bundle structure over the projective line.

\begin{proposition}[\cite{NBOOK}]\label{lem:SpPl}
 The analytic structure on $(X,0)$ carries  a weighted homogeneous
 structure.
 Moreover, the minimal good resolution of any
weighted homogeneous singularity with
Seifert invariants  $(b_0,g=0;\{(\alpha_j,\omega_j)\}_j)$
admits such an analytic plumbing representation for
certain constants $\{p_j\}_j$.
By Theorem \ref{th:ACTh} we can even assume that each $p_j$ is non--zero
(what we will assume below).
\end{proposition}

The $\C^*$-orbits lifted to $\widetilde{X^a}$ and closed are as follows:
the generic ones, which intersect $E_{0}$ sit in $U_0\cup U_{-1}$ and
are given by $\{v_0=c\}$, where 
$c\in(\C\setminus \{\cup_j\{p_j\}\})\cup \infty$. The special Seifert orbit for each $j$ in $U_{j,s_j}$ is given by $\{v_{j,s_j}=0\}$.

On the $U_0$ chart for generic $v_0$ the $\C ^*$-action can be described as $t* (u_0,v_0) = (t \cdot u_0, v_0)$.

\subsubsection{\bf Description of a basis for $H^0(\calO_{\tX}( K + [Z_K]))/H^0(\calO_{\tX}(K))$. } \label{ss:whbasis}

In the case of weighted homogenous singularities a basis for $H^0(\calO_{\tX}( K + [Z_K]))/H^0(\calO_{\tX}(K)) = H^1(\calO_{\tX})^*$ can be explicitely described. In the sequel we preent this description following \cite[12.2]{NNA1}.

First of all, by Laufer duality from \cite{Laufer72} one can identify the dual space $H^1(\calO_{\tX})^*$ with the space of global 
holomorphic 2-forms on $\widetilde{X}\setminus E$ modulo the subspace of those forms which can be extended holomorphically over $X$. Therefore, we will describe a basis for $H^0(\widetilde{X}\setminus E,\Omega_{\widetilde{X}}^2)/H^0(\widetilde{X},\Omega_{\widetilde{X}}^2)$.

For $\ell,\, \{m_j\}_j\in \Z$, $n\in\Z_{\geq 0}$, we consider the form  $\omega_{\ell,n}^{0}:= u_0^{-\ell-1}\prod_j
(v_0-p_j)^{-m_j}v_0^ndv_0\wedge du_0$ viewed as a section of
$\calO_{\tX}(K)$ over $U_0$, with possible poles over $E\cap U_0$. 
This, under the transformation $v_0=u_{-1}^{-1}$, $u_0=u_{-1}^{b_0}v_{-1}$, transforms
into the following form on $U_{-1}$:
$$\pm\,u_{-1}^{-b_0\ell+\sum m_j-n-2}
v_{-1}^{-\ell-1}\textstyle{\prod_j} (1-u_{-1}p_j)^{-m_j}du_{-1}
\wedge dv_{-1}.$$
Furthermore, the regularity over $\tX\setminus E$ 
implies  that the exponent of $u_{-1}$ is non-negative, hence 
\begin{equation}\label{eq:whforms1}
n\leq - b_0\ell-2 +\textstyle{\sum_j}m_j.
\end{equation}

If we fix one of the legs, say $j$, by induction using substrings of the legs and the corresponding continued
fraction identities (facts used intensively in cyclic quotient
invariants computations) one gets that the transformation between chart $U_0$ and
$U_{j,s_j}$ is $u_0=u_{j,s_j}^{-\tau_j} v_{j,s_j}^{-\omega_j}$, $v_0=u_{j,s_j}^{\omega_j'}v_{j,s_j}^{\alpha_j}$. 
Then, $\omega_{\ell,n}^{0}$ in the chart $U_{j,s_j}$ under this  transformation becomes
$$
u_{j,s_j}^{\tau_j\ell -\omega_j'm_j+\omega_j'-1}
v_{j,s_j}^{\omega_j\ell -\alpha_jm_j+\alpha_j-1}
(u_{j,s_j}^{\omega_j'} v_{j,s_j}^{\alpha_j}+p_j)^n \cdot
\textstyle{\prod _{j'\not=j}}
(u_{j,s_j}^{\omega_j'} v_{j,\s_j}^{\alpha_j}+p_{j'}-p_j)^{-m_j}\, dv_{j,s_j}
\wedge du_{j,s_j}.
$$
Again, by the regularity over $\tX\setminus E$, the exponent of
$v_{j,s_j}$ should be non-negative, hence one gets the inequality 
$$\omega_j\ell -\alpha_jm_j+\alpha_j-1\geq 0,$$
for which the largest solution for $m_j$ is 
\begin{equation}
m_j=   \left \lceil{\omega_j\ell/\alpha_j}\right \rceil.
\end{equation}

In the sequel let us fix this maximal value for $m_j$. Thus, the form $\omega_{\ell,n}^{0}$ extends to a form $\omega_{\ell,n}$
on $\tX$ regular on $\tX\setminus E$, if for $m_j:=   \left \lceil{\omega_j\ell/\alpha_j}\right \rceil$ (for all $j$)
the inequality (\ref{eq:whforms1}) holds.

If $\ell< 0$ then $m_j=  \left \lceil{\omega_j\ell/\alpha_j}\right \rceil \leq 0$, hence
the form $\omega_{\ell,n}$ is regular on  $\tX$, hence in
$H^0(\calO_{\tX}( K + [Z_K]))/H^0(\calO_{\tX}(K))$ it is zero. Therefore, we consider only the values $\ell \geq 0$ and associated with them we set:

$$n_{\ell}:= - b_0\ell-2 +\textstyle{\sum_j}\left \lceil{\omega_j\ell/\alpha_j}\right \rceil.$$
If $n_{\ell}<0$ then there is no such form with pole $\ell+1$
along $E_{0}$, cf. (\ref{eq:whforms1}). Hence, we consider the set $\calw:=\{\ell\geq 0\,:\, n_{\ell}\geq 0\}$. Notice that if $\ell = 0$, then $n_{\ell} = -2$, which means that $\calw$ contains only positive integers. Then, one can prove the following result:

\begin{lemma}[\cite{NBOOK,NNA1}]\label{lem:whforms}
The forms $\omega_{\ell,n}$ associated with any $\ell\in\calw$ and $0\leq n\leq n_{\ell}$ 
form a basis of $H^0(\calO_{\tX}( K + [Z_K]))/H^0(\calO_{\tX}(K))$.
\end{lemma}

\subsection{The dominance} We note that one has a $\C^*$-action on the resolution space $\tX$, which induces a linear action on $\pic^{l'}(\tX)$ for all $l' \in -\calS'$ by $t^* \calL = f_t^*(\calL)$, where $f_t$ is the action map of $t \in \C^*$ on the resolution $\tX$. 
This $\C^*$-action maps the union of exceptional divisors $E$ into itself, so if $l \in L$ then $t^*(\calO_{\tX}(-l)) = \calO_{\tX}(-l)$. Hence, by linearity of the action $t^*$ follows
that $t^*(\calO_{\tX}(-l')) = \calO_{\tX}(-l')$. In other words, the natural line bundle $\calO_{\tX}(-l')$ is a fix point of the $\C^*$-action on $\pic^{-l'}(\tX)$.

On the other hand, there is a similar $\C^*$-action on the dual space $\pic^0(\tX)^* = H^0(\calO_{\tX}( K + [Z_K]))/H^0(\calO_{\tX}(K))$, which is given by $t^* [\omega] = [f_t^*(\omega)]$.

Recall that the duality between $H^0(\calO_{\tX}( K + [Z_K]))/H^0(\calO_{\tX}(K))$ and $H^1(\calO_{\tX})$ is given by Laufer duality, and consider the following situation. 

Let $D$ be a transversal cut along an exceptional divisor $E_u$, and $Z\gg 0$ be a very big cycle.
Let $(x, y)$ be local coordinates of $\tX$ near the intersection point $E_u \cap D$ such that $E_u = \{x = 0\}$ and $D = \{y = 0\}$. 
Let us realise a tangent vector $\mathrm{v} \in T_0(\pic^0(\tX)) \cong \pic^0(\tX)$ by an aproppriate deformation of the divisor $D$ given by the form $g(s) = [y + s \cdot \sum_{0 \leq k \leq Z_u -1} a_k \cdot x^k]$, and use the notation $g(s) = D_s$.

We can express a representative of  an element $w \in H^0(\calO_{\tX}( K + [Z_K]))/H^0(\calO_{\tX}(K))$ by a differential form $\omega$ in local cordinates as $\omega = (\sum_{1 \leq i, -Z_u \leq j} a_{i, j} y^i x^j) dx \wedge dy$, so by Laufer integration formula we get:

\begin{equation*}
\omega(v) = \frac{d}{ds}\left(  \int_{|x|=\epsilon, \atop |y|=\epsilon}  \log \left(1+ s  \cdot \frac{ \sum_{0 \leq k \leq Z_u -1} a_k \cdot x^k}{y} \right) \left(\sum_{1 \leq i, -Z_u \leq j} a_{i, j} y^i x^j \right) dx\wedge dy  \right).
\end{equation*}

On the other hand, we can represent $t^*(v)$ as the tangent vector $t^*(v) \in T_0(\pic^0(\tX)) \cong \pic^0(\tX)$ of the deformation $f_t^*(D_s)$ of the divisor $f_t^*(D)$  and $t^*(w)$ by $[f_t^*(\omega)]$. However, notice that in the coordinates $x' = f_t \circ x$ and $y' = f_t \circ y$ we can write $f_t^*(D_s) = [y' + s \cdot \sum_{0 \leq k \leq Z_u -1} a_k \cdot x'^k]$ and $f_t^*(\omega) = (\sum_{1 \leq i, -Z_u \leq j} a_{i, j} y'^i x'^j) dx' \wedge dy'$, which means that $t^*(w)(t^*(v)) = w(v)$.
Since these kind of tangent vectors generate $T_0(\pic^0(\tX))$, we deduce that the two $\mathbb{C}^*$-actions on $H^0(\calO_{\tX}( K + [Z_K]))/H^0(\calO_{\tX}(K))$ and $\pic^0(\tX)$
preserve the pairing between the two vector spaces.

We can write this $\mathbb{C}^*$-action on the vector space $H^0(\calO_{\tX}( K + [Z_K]))/H^0(\calO_{\tX}(K))$ in the basis $\omega_{\ell,n}$ as follows:

\begin{equation}
t * \omega_{\ell,n} =  t^{-\ell} \cdot u_0^{-\ell-1}\prod_j (v_0-p_j)^{-m_j}v_0^n   dv_0\wedge du_0 = t^{-\ell} \cdot \omega_{\ell,n}.
\end{equation}

Now, we know that if $\ell\in\calw$ then $\ell > 0$, which implies that the linear action of $\C^*$ on $\pic^0(\tX)^*$ has only negative weights, so the linear action of $\C^*$ on $\pic^0(\tX)$ has only positive weights.

After all these preparations, we can state and prove the main theorem of this section:

\begin{theorem}
Let $(X, 0)$ be a weighted homogeneous singularity with resolution $\tX$ and consider a cycle $l' \in \calS'$.  Then for every line bundle $\calL \in \pic^{-l'}(\tX)$ one has $$h^1(\tX, \calL) \leq h^1(\calO_{\tX}(-l')).$$
\end{theorem}
\begin{proof}

By the previous discussions we know that the natural line bundle $\calO_{\tX}(-l')$ is a fixed point of the $\C^*$-action on $\pic^{-l'}(\tX)$.

It means that the action of $\C^*$ on $\pic^0(\tX)$ is the same as the action of $\C^*$ on $\pic^{-l'}(\tX)$ if we look at the line bundle $\calO_{\tX}(-l')$ as the origin of the affine complex space $\pic^{-l'}(\tX)$.

Notice that if we have a line bundle $\calL \in \pic^{-l'}(\tX)$ and an arbitrary $t \in \C^*$, then trivially $h^1(\tX, t * \calL) = h^1(\tX, \calL)$.  
Furthermore, the linear action of $\C^*$ on $\pic^{-l'}(\tX)$ has only positive weights, which means that the line bundle $\calO_{\tX}(-l')$ is in the closure of all orbits  $\overline{\C^* * \calL}$.  Then, by semicontinuity of $h^1$,  we indeed get that $h^1(\tX, \calL) \leq h^1(\calO_{\tX}(-l'))$. 

\end{proof}

\section{Examples for wild properties of counting functions for the topological Poincar\'e series in the non-splice quotient case} \label{s:wildprop}

\subsection{Upper bounds in the splice quotient case} Recall  that if a resolution graph $\mathcal{T}$ satisfies the monomial conditions, then there exists a splice quotient analytical structure on $\mathcal{T}$, and for each splice quotient analytical structure $\tX$ one has $h^1_{virt}(l') = h^1(\calO_{\tX}(-l'))$ for every $l' \in \calS'$. 

In this case, we consider the computation sequence  $x_0 = 0$, $x_{i+1} = x_i + E_{v_i}$ for some $v_i \in \calv$, $1 \leq i \leq N-1$ and $x_N = \lfloor Z_K\rfloor$.  Then we have the following exact sequences $H^1(\calO_{E_{v_i}}( -l'- x_i)) \to  H^1(\calO_{x_{i+1}}( -l')) \to H^1(\calO_{x_{i}}( -l')) \to 0$. This implies that  $ h^1(\calO_{x_{i+1}}( -l')) \leq h^1(\calO_{E_{v_i}}( -l'- x_i)) + h^1(\calO_{x_{i}}( -l')) =  h^1(\calO_{x_{i}}( -l')) +\max (0, -1 +( l' + x_i, E_{v_i}))$.

Furthermore,  the exact sequence $H^1(\calO_{\tX}( -l' - x_N)) \to  H^1(\calO_{\tX}( -l')) \to H^1(\calO_{x_{N}}( -l')) \to 0$ together with the vanishing  $h^1(\calO_{\tX}( -l' - x_N))=0$ given by the Grauert-Riedemschnieder theorem, imply  that $ h^1(\calO_{\tX}( -l')) =  h^1(\calO_{x_{N}}( -l'))$. This gives the following inequality in the splice quotient case
\begin{equation*}\label{eq:upb}
h^1_{virt}(l') = h^1(\calO_{\tX}(-l')) \leq \sum_{1 \leq i \leq N-1}  \max(0, -1 +( l' + x_i, E_{v_i})).
\end{equation*}

It implies that $h^1_{virt}(l') \leq N \cdot (1 + \sum_{v \in \calv}(Z_K)_v) \leq (\sum_{v \in \calv}(Z_K)_v + |\calv|) \cdot (1 + \sum_{v \in \calv}(Z_K)_v)$.

We would like to emphasize that this idea gives rise to the path-lattice cohomological upper bound for the geometric genus, developed by N\'emethi in \cite{NJEMS}.

Nevertheless, in the sequel we will show that this type of bounds are very far from being true in the general case, and this is really a specialty of the graphs satisfying the monomomial condition. Philosopically speaking, this will show that while in many  calculations the toplogical Poincar\'e series and its combinatorial model behaves as nice as the analytic Poincar\'e series in the splice quotient case, there are still some novelties coming up.

\subsection{} We first construct a resolution graph $\mathcal{T}''$ with a vertex $v$ and a cycle $l' \in \calS'$, such that $(l', E_v) = 0$ and $\sum_{l \in L, l \geq 0, l_v = 0} z^{\mathcal{T}''}(l' + l) \notin \{-1, 0, 1\}$.

Recall that in section \ref{ss:cex} we have constructed the graph $\mathcal{T}_{new}$ as the blown up at vertex $v$ of the graph $\calt$ shown on the left hand side of Figure \ref{fig:tt1}. We take two pieces of it, say $\mathcal{T}_{new, 1}$ and $\mathcal{T}_{new, 2}$, and glue together at their corresponding vertices $v_{new,1}$ and $v_{new,2}$.  Let us denote this vertex by $w$ and associate with it a self-intersection number $-K$, where $K$ is very large. We also denote the neighbouring vertices of $w$ by $v_1$ and $v_2$.

Then $\calt''$ is created by blowing up the above constructed graph at the vertex $w$ and denote the newly created vertex by $w'$. For the cycle $l' = E_{w}^* + E_{v_1}^* + E_{v_2}^*\in L'(\calt'')$ we observe the followings.

\begin{lemma}
\begin{equation}
\sum_{l \in L, l \geq 0, l_{w'} = 0} z^{\mathcal{T}''}(l' + l) = \sum_{l \in L, l \geq 0, l_{w'} = 0, l_{w} = 0} z^{\mathcal{T}''}(l'+ l) = -3.
\end{equation}
\end{lemma}
\begin{proof}
Indeed, one can see that 
\begin{equation*}
\sum_{l \in L, l \geq 0, l_{w'} = 0} z^{\mathcal{T}''}(l' + l) = \sum_{l \in L, l \geq 0, l_{w'} = 0, l_{w} = 0} z^{\mathcal{T}''}(l'+ l) = 2CD - C^2,
\end{equation*}
where we have $C = \sum_{l \geq 0, l_{new} = 0} z^{\calt_{new}}(E_v^* + l)$ and $D = \sum_{l \geq 0, l_{new} = 0, l_v = 0} z^{\calt_{new}}((M+1) \cdot E_v^* - E_{n_1}^* - E_{n_2}^* + l)$.

It means we only have to prove that $D = 1$ if the selfintersection number $-M$ is low enough.

We introduce  $F_j = \sum_{l \in L, l \geq 0, l_{v_{1, 5}} = 0, l_{n_1} = 0} z^{\calt_{1}}( (3j-1) E_{n_1}^* - j(\sum_{1 \leq i \leq 4} E_{v_i}^*)+ l)$ for any $j \geq 0$. Then one has

\begin{equation*}
D = \textnormal{Coeff} \Big(  \big(\sum_{0 \leq j} F_j \cdot x^j \big)^2 \cdot (1-x) ,x^{M+1} \Big)
\end{equation*}
Now the claim follows instantly if we can show that $F_j = 1$ for $j$ large enough.

Notice that $F_j = \sum_{l \in L, l \geq 0, l_{v_{1, 5}} = 0, l_{n_1} = 0} z^{\calt_{1}}( (3j-1 - 4z) E_{n_1}^* + (Nz- j) \cdot (\sum_{1 \leq i \leq 4} E_{v_i}^*)+ l)$, where $z = \left \lceil{\frac{j}{N}}\right \rceil$, since $(3j-1 - 4z) E_{n_1}^* + (Nz- j) \cdot (\sum_{1 \leq i \leq 4} E_{v_i}^*)$ is the smallest cycle in $\calS'$ which is larger than $(3j-1) E_{n_1}^* - j(\sum_{1 \leq i \leq 4} E_{v_i}^*)$.

Let us consider $l''  = (3j-1 - 4z) E_{n_1}^* + (Nz- j) \cdot (\sum_{1 \leq i \leq 4} E_{v_i}^*)$ and notice that $l'' \in \calS'$ and let $\calt_b$ be the graph resulted by blowing down the vertex $v_5$ in $\calt_1$. Then we have

\begin{equation*}
F_j = h^1_{virt}(\calO_{\tX_1}(-l''  - E_{v_5})) - h^1_{virt}(\calO_{\tX_b}(-\pi_*(l''))).
\end{equation*}

On the other hand we know that $\calt_1$ satisfies the monomial conditions, thus there is a splice quotient structure $\tX_1$ on it, which implies that

\begin{equation*}
F_j = h^1(\calO_{\tX_1}(-l''  - E_{v_5})) - h^1(\calO_{\tX_b}(-\pi_*(l''))).
\end{equation*}

Finally $F_j = 1$ comes from the fact that the line bundle $\calO_{\tX_b}(-\pi_*(l''))$ has no base point on the exceptional divisor $E_{n_1}$.
\end{proof}

Now we will use the lemma above to construct an example of its own interest.

We glue together $t$ pieces of the graph $\mathcal{T}''$ constructed above at their common  vertex $w'$.  We give to $w'$ a selfintersection number $- t \cdot C$, where $C$ is a constant and if $t$ is enough large, then the resulted resolution graph $\mathcal{T}_t$ is negative definite. One can see that such a constant $C$ exists.

Then, for the cycle $l'_t = \sum_{1 \leq i \leq t}l'_i = \sum_{1 \leq i \leq t}(  E_{w_i}^* + E_{v_{1, i}}^* + E_{v_{2, i}}^*  )$ considered on $\mathcal{T}_t$ one proves the following expression.

\begin{lemma}
\begin{equation*}
\sum_{l \in L, l \geq 0, l_{w'} = 0} z^{\calt_t}(l'_t + l) = \left(\sum_{l \in L, l \geq 0, l_{w'} = 0, l_{w} = 0} z^{\calt'}(l'+ l)  \right)^t = (-3)^t.
\end{equation*}
\end{lemma}
\begin{proof}
We know that the neighbours of the vertex $w'$ in the resolution graph $\mathcal{T}_t$ are $w_1, \cdots, w_t$. Thus we have
\begin{equation*}
\sum_{l \in L, l \geq 0, l_{w'} = 0} z^{\calt_t}(l'_t+ l)  = \sum_{l \in L, l \geq 0, l_{w'} = 0, l_{w_i} = 0, 1 \leq i \leq t} z^{\calt_t}(l'_t + l).
\end{equation*}

\begin{equation*}
\sum_{l \in L, l \geq 0, l_{w'} = 0} z^{\calt_t}(l'_t + l)  = \sum_{l \in L, l \geq 0, l_{w'} = 0, l_{w_i} = 0} \prod_{1 \leq i \leq t}z^{\calt'_i}((\pi_i)_*( l'_t + l) ).
\end{equation*}

\begin{equation*}
\sum_{l \in L, l \geq 0, l_{w'} = 0} z^{\calt_t}(l'_t + l) = \left(\sum_{l \in L, l \geq 0, l_{w'} = 0, l_{w} = 0} z^{\calt'}(l'+ l)  \right)^t = (-3)^t.
\end{equation*}
\end{proof}

The expression from the lemma yields $h^1_{virt}(\calO_{\tX}(-l'_t)) - h^1_{virt}(\calO_{\tX}(-l'_t-E_{w'})) - \chi(l'_t) + \chi(l'_t + E_{w'}) =  (-3)^t$, so $h^1_{virt}(\calO_{\tX}(-l'_t)) - h^1_{virt}(\calO_{\tX}(-l'_t-E_{w'})) \geq  (-3)^t  + (l'_t, E_{w'}) - 1 =  (-3)^t - 1$.

It means that $\max( |h^1_{virt}(\calO_{\tX}(-l'_t)) |,  | h^1_{virt}(\calO_{\tX}(-l'_t-E_{w'}))| )  \geq \frac{|(-3)^t - 1|}{2}$.

On the other hand, one can see that all the quantities in the expressions like $(\sum_{v \in \calv_t}((Z_K)_t)_v + |\calv_t|) \cdot (1 + \sum_{v \in \calv_t}((Z_K)_t)_v)$ grow at most polynomially in $t$ which gives the desired counterexample mentioned at the beginning of the section.

Another thing which can be noticed is that when the resolution graph satisfies the monomial conditions, and we have a splice quotient analytical structure $\tX$ supported on it, for a fixed cycle $l' \in \calS'$ and a vertex $v \in \calv$  one has $\sum_{l \in L, l \geq 0, l_v = 0}z^{\mathcal{T}}(l' + l) = \dim\left(\frac{H^0(\tX, \calO_{\tX}( -l'))}{H^0(\tX, \calO_{\tX}( -l'-E_v))}\right)$,
which yields $\sum_{l \in L, l \geq 0, l_v = 0}z^{\mathcal{T}}(l' + l) \geq 0$. However, we have see that this nonnegativity is also false in the general case, because $\sum_{l \in L, l \geq 0, l_{w'} = 0} z^{\calt_t}(l'_t + l) =  (-3)^t$ is sometimes negative.

\subsection{Negativity of the periodic constant} Notice that when the monomial conditions hold for a resolution graph, then for the corresponding splice quotient singularity the geometric genus coincides with the canonical normalized Seiberg-Witten invariant, which is therefore nonnegative.

On the other hand with very similiar ideas as before one can recursively construct resolution graphs such that the periodic constants or even the canonical normalised Seiberg-Witten invariant is negative.

We are thankful to Andr\'as N\'emethi who suggested the following graph

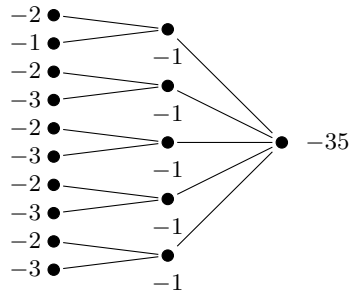
\begin{center}
\begin{tikzpicture}[scale=.75]

\node (v0) at (5,0) {};
\draw[fill] (5,0) circle (0.1);
\node at (5.8,0) {\small $-35$};

\node (v1) at (3,2) {};
\draw[fill] (3,2) circle (0.1);
\node at (3,1.5) {\small $-1$};
\node (v11) at (1,2.25) {};
\draw[fill] (1,2.25) circle (0.1);
\node at (0.5,2.25) {\small $-2$};
\node (v12) at (1,1.75) {};
\draw[fill] (1,1.75) circle (0.1);
\node at (0.5,1.75) {\small $-1$};

\node (v2) at (3,1) {};
\draw[fill] (3,1) circle (0.1);
\node at (3,0.5) {\small $-1$};
\node (v21) at (1,1.25) {};
\draw[fill] (1,1.25) circle (0.1);
\node at (0.5,1.25) {\small $-2$};
\node (v22) at (1,0.75) {};
\draw[fill] (1,0.75) circle (0.1);
\node at (0.5,0.75) {\small $-3$};

\node (v3) at (3,0) {};
\draw[fill] (3,0) circle (0.1);
\node at (3,-0.5) {\small $-1$};
\node (v31) at (1,0.25) {};
\draw[fill] (1,0.25) circle (0.1);
\node at (0.5,0.25) {\small $-2$};
\node (v32) at (1,-0.25) {};
\draw[fill] (1,-0.25) circle (0.1);
\node at (0.5,-0.25) {\small $-3$};

\node (v4) at (3,-1) {};
\draw[fill] (3,-1) circle (0.1);
\node at (3,-1.5) {\small $-1$};
\node (v41) at (1,-0.75) {};
\draw[fill] (1,-0.75) circle (0.1);
\node at (0.5,-0.75) {\small $-2$};
\node (v42) at (1,-1.25) {};
\draw[fill] (1,-1.25) circle (0.1);
\node at (0.5,-1.25) {\small $-3$};

\node (v5) at (3,-2) {};
\draw[fill] (3,-2) circle (0.1);
\node at (3,-2.5) {\small $-1$};
\node (v51) at (1,-1.75) {};
\draw[fill] (1,-1.75) circle (0.1);
\node at (0.5,-1.75) {\small $-2$};
\node (v52) at (1,-2.25) {};
\draw[fill] (1,-2.25) circle (0.1);
\node at (0.5,-2.25) {\small $-3$};

\draw  (v0) edge (v1);
\draw  (v0) edge (v2);
\draw  (v0) edge (v3);
\draw  (v0) edge (v4);
\draw  (v0) edge (v5);

\draw  (v1) edge (v11);
\draw  (v1) edge (v12);

\draw  (v2) edge (v21);
\draw  (v2) edge (v22);

\draw  (v3) edge (v31);
\draw  (v3) edge (v32);

\draw  (v4) edge (v41);
\draw  (v4) edge (v42);

\draw  (v5) edge (v51);
\draw  (v5) edge (v52);

%
%
%
%
%

\end{tikzpicture}
\end{center}
For more details on the above graph we refer to the articles \cite{BodN,LSzalg}. These also contain the appropiate methods for computing the canonical normalized Seiberg-Witten invariant, which is in fact $-6$ in this case.

\section{monomial conditions and counting functions of the topological Poincar\'e series.} \label{s:mctPs}

Note that by section \ref{s:wildprop}, if a resolution graph $\mathcal{T}$ does not satisfy the monomial conditions and $l' \in L'$ then $\sum_{l \in L, l \geq 0, l_v = 0}z^{\mathcal{T}}(l' + l)$
can be even negative or very large.

Contrary, if  $\mathcal{T}$ satisfies the monomial conditions, and we consider its associated splice quotient singularity and a fixed cycle $l'\in L'$, then by (\ref{eq:sqcount}) one has 
$$\sum_{l \in L, l \geq 0, l_v = 0}z^{\mathcal{T}}(l' + l) = \dim\left(\frac{H^0(\tX, \calO_{\tX}( -l'))}{H^0(\tX, \calO_{\tX}( -l'-E_v))}\right),$$ which gives the following bounds 
$0 \leq \sum_{l \in L, l \geq 0, l_v = 0}z^{\mathcal{T}}(l' + l)  \leq -(l', E_v) + 1$. 

Despite that we know this statement from the analytic theory, the aim of this section is to provide an alternative, purely combinatorial proof in order to shed a light on how the monomial conditions control the counting functions and the monomial cycles create the above mentioned bounds in case of their existence.

In the following we consider a resolution graph $\mathcal{T}$. Recall that for any vertex  $w$ the connected components of $\Gamma\setminus w$ are called the {\it branches} of $w$.

Assume  $\mathcal{T}$ satisfies the monomial conditions and  we fix a vertex $v \in \calv$.     
Denote the neighbours of the vertex $v$ by $u_1, \cdots, u_t$ and let $n_i$ be the first node on the chain starting from $v$ in the direction towards $u_i$.
For any vertex  $w \neq v$, its neighbour contained in the same branch as $v$ will be denoted by $s_w$, and the other neighbours  by $u_{w, j}, 1 \leq j \leq \delta_w-1$. For such a vertex $w \neq v$ we denote by $\mathcal{T}_w$ the subgraph containing $w$ and the branches of $w$ which do not contain $v$, and let $\calv_w$ be its set of vertices.

Then,  if $l' \in L'$ is a cycle, one can consider its restriction to $L'_{\mathcal{T}_w}$ and it will be denoted by $r_w(l')$.
Associated with a vertex $w$ one considers the series 
$$R^{\mathcal{T}_w}(\bt_{\calv_w}) =  Z^{\mathcal{T}_w}(\bt_{\calv_w})  \cdot ( 1 - \bt_{\calv_w}^{r_w(E_w^*)})=\sum_{l''\in L'_{\mathcal{T}_w}} r^{\mathcal{T}_w}(l'')\bt_{\calv_w}^{l''}.$$
Then, as a starting point, we prove the following proposition.

\begin{proposition}\label{segedprop}
(1) For any  node $n \neq v$ and a cycle $l' \in L'$ we have $$\sum_{0 \leq l \in L_{\calv_n}, l_n = 0} r^{\mathcal{T}_n}( r_n(l') + l ) \in \{0, 1\}$$.

(2) Let us consider a vertex $w \neq v$ such that $s_w$ is a node. Then for any $l' \in L'$ we have $$\sum_{0 \leq l \in L_{\calv_w}, l_w = 0} r^{\mathcal{T}_w}( r_w(l') + l ) \in \{0, 1\}.$$ 
Furthermore, if $\sum_{0 \leq l \in L_{\calv_w}, l_w = 0} r^{\mathcal{T}_w}( r_w(l') + l ) = 1$ then we also have $\sum_{0 \leq l \in L_{\calv_w}, l_w = 0} r^{\mathcal{T}_w}( r_w(l' + E_w) + l ) = 1$. 
\end{proposition}

\begin{proof}

For simplicity, associated with a vertex $w \neq v$ and $l' \in L'$ we introduce the notation $$C(w, l'):= \sum_{0 \leq l \in L_{\calv_w}, l_w = 0} r^{\mathcal{T}_w}( r_w(l') + l ).$$

We will proceed the proof by simultaneous downgoing induction on the distance of the vertex $w$ from the fixed vertex $v$. The induction step will also contain the base case.

\underline{First we prove the induction step for part (1) as follows.}

Let $n \neq v$ be a node with neighbours $u_{n, j}, s_n$ ($1 \leq j \leq \delta_n-1)$ and a cycle $l' \in L'$. For any $1 \leq j \leq \delta_n-1$ we introduce the series $g_j(x) = \sum_{i \geq 0} C(u_{n, j}, l' + i E_{u_{n, j}}) x^i$, then we have:

\begin{equation*}
C(n, l') = \textnormal{ Coeff} \left( (1- x)^{\delta_n - 2} \cdot \prod_{1 \leq j \leq \delta_n -1} g_j(x) ,   x^{-(l', E_n)} \right). 
\end{equation*} 
Indeed, for this we have to show that
\begin{equation*}
C(n, l') =    \sum_{k_j\geq 0 \atop k:= \sum k_j \leq -(l', E_n)} (-1)^{ -(l', E_n) - k}   \cdot {\delta_n - 2 \choose  -(l', E_n) - k} \cdot \prod_{ 1 \leq j \leq \delta_n-1}  \sum_{0 \leq l \in L_{\calv_{u_{n, j}}} \atop l_{u_{n, j}} = 0} r^{\mathcal{T}_{u_{n, j}}}( r_{u_{n, j}}(l') +  i E_{u_{n, j}} +  l ) . 
\end{equation*} 
For the integers $0 \leq k_j , 1 \leq j \leq \delta_n-1$ such that $k = \sum k_j \leq -(l', E_n)$ let us introduce the following
 $$C(n, l', k_1, \cdots, k_{\delta_n - 1}) = \sum_{0 \leq l \in L_{\calv_n}, l_n = 0, l_{u_{n, j}} = k_j} r^{\mathcal{T}_w}( r_w(l') + l ).$$
We know that $\sum_{0 \leq k_j , 1 \leq j \leq \delta_n-1, k = \sum k_j \leq -(l', E_n)} C(n, l', k_1, \cdots, k_{\delta_n - 1}) = C(n, l')$ and for a certain choice of the integers
$0 \leq k_j , 1 \leq j \leq \delta_n-1$ one has

\begin{equation*}
C(n, l', k_1, \cdots, k_{\delta_n - 1})=  (-1)^{ -(l', E_n) - k}   \cdot {\delta_n - 2 \choose  -(l', E_n) - k} \cdot \prod_{ 1 \leq j \leq \delta_n-1}  \sum_{0 \leq l \in L_{\calv_{u_{n, j}}} \atop l_{u_{n, j}} = 0} r^{\mathcal{T}_{u_{n, j}}}( r_{u_{n, j}}(l') +  i E_{u_{n, j}} +  l ),
\end{equation*} 
which proves the formula completely.

The induction hypothesis assumes part (2) for the vertices $u_{n, j}$, that is $C(u_{n, j}, l' + i E_{u_{n, j}})\in\{0,1\}$ and if $C(u_{n, j}, l' + i E_{u_{n, j}})=1$ then $C(u_{n, j}, l' + (i+1) E_{u_{n, j}})=1$ too.  Hence we get the expression $g_j(x) = \frac{x^{m_j}}{1-x}$ for some integer $m_j \geq 0$. If we set $m:= \sum_{1 \leq j \leq \delta_n-1} m_j$, then we have

\begin{equation*}
C(n, l') = \textnormal{ Coeff} \left(  \frac{x^{m}}{1-x} ,  x^{-(l', E_n)} \right).
\end{equation*}

This implies that $C(n, l') = 1$ if $m \leq -(l', E_n)$ and $C(n, l') = 0$ if $m > -(l', E_n)$, and the induction step is finished for part (1).

\underline{Next, we continue with the induction step for the first statement of part (2).}

Consider a vertex $w \neq v$ such that $s_w$ is a node and $l' \in L'$. In fact we will prove that if $w' \in \calv_w$ then one has $C(w', l') \in \{0, 1\}$.

First of all, note that if $w'$ is an end-vertex, then the statement is trivial. If $w'$ is a node then this follows from part (1). 

If $w'$ is a vertex with valency $2$ and its neighbours are $s_{w'}$ and $u_{w'}$ then one has the equality $C(w', l') = C(u_{w'}, l' - (l', E_{w'}) E_{u_{w'}})$.
If $u_{w'}$ is a node then we are done by part (1). If $\delta_{u_{w'}} = 2$ and the other neighbour is $u_{u_{w'}}$ then we can continue this process until we reach a node or an end-vertex $\widetilde{w}$ with $C(w, l') = C(\widetilde{w}, \widetilde{l})$ for some  $\widetilde{l}$.

\underline{Now we prove the second part of (2).}

Let $Z \geq 0$ be a monomial cycle associated with the node $s_w$, such that  $Z$ is supported on the branch of $s_w$ containing the vertex $w$,  and if $(Z, E_w) < 0$
then $w$ is an end-vertex on this branch.
Let us write $Z = \sum_{ {w'} \in \calv_w} Z_{w'} \cdot E_{w'}$ with the convention $Z_{s_{w}} = 0$. 

We will prove by downgoing induction on the distance from vertex $w$ that for any vertices $w' \in \calv_w$  and $l'$ such that $C(w', l') = 1$ one has 
$C(w', l' + Z_{w'} \cdot E_{w'} + Z_{s_{w'}} \cdot E_{s_{w'}}) = 1$. In particular, this applied for  $w':= w$ would give us the  second statement of (2).

Notice that if $w'$ is an end-vertex then the statement is trivial since in this case $(l' + Z_{w'} \cdot E_{w'} + Z_{s_{w'}} \cdot E_{s_{w'}}, E_{w'}) \geq (l', E_{w'})$. 
If $\delta_{w'} = 2$ and its neighbours are $s_{w'}$ and $w''$ then we have $C(w', l') = C(w'', l' - (l', E_{w'}) E_{w''})$.
By setting $l'' := l' + Z_{w'} \cdot E_{w'} + Z_{s_{w'}} \cdot E_{s_{w'}}$, then we have $C(w', l'') = C(w'', l'' - (l'', E_{w'}) E_{w''} )$.

Notice that we have the equation $ l'' - (l'', E_{w'}) E_{w''}  = l' - (l', E_{w'}) E_{w''} + Z_{w'} \cdot E_{w'} + Z_{s_{w'}} \cdot E_{s_{w'}} - Z_{s_{w'}} \cdot  E_{w''}  - Z_{w'} (E_{w'}, E_{w'}) \cdot  E_{w''}$.          
On the other hand, $(Z, E_{w'}) = 0$, so we have $- Z_{s_{w'}}  - Z_{w'} (E_{w'}, E_{w'}) = Z_{w''}$, which means that $C(w', l'') = C(w'', l' - (l', E_{w'}) E_{w''} + Z_{w'} \cdot E_{w'} + Z_{w''} \cdot E_{w''})$.

The induction hypothesis assumes that $C(w'', l' - (l', E_{w'}) E_{w''} + Z_{w'} \cdot E_{w'} + Z_{w''} \cdot E_{w''}) \geq C(w'', l'' - (l'', E_{w'}) E_{w''} )$. Hence we get that $C(w', l'') \geq C(w', l')$,  
which proves the statement in this case too.

Assume finally that $w'$ is a node with neighbours $s_{w'}$, $u_{w', j}, 1 \leq j \leq \delta_{w'}-1$. For $1 \leq j \leq \delta_{w'}-1$ we set $g_j(x) := \sum_{i \geq 0} C(u_{w', j}, l' + i E_{u_{w', j}}) x^i$.  As in the proof of part (1) we have

\begin{equation*}
C(w', l') = \textnormal{ Coeff} \left( (1- x)^{\delta_{w'} - 2} \cdot \prod_{1 \leq j \leq \delta_{w'} -1} g_j(x) ,   x^{-(l', E_{w'})} \right).
\end{equation*}

Then, using the induction hypothesis about part (2) for the vertices $u_{w', j}$ we get that $g_j(x) = \frac{x^{m_j}}{1-x}$ for some integers $m_j \geq 0$, so we have
$C(w', l') = 1$ if $\sum_{1 \leq j \leq \delta_{w'}-1} m_j \leq -(l', E_{w'})$ and $C(w', l') = 0$ if $\sum_{1 \leq j \leq \delta_{w'}-1} m_j > -(l', E_{w'})$.

We consider the cycle $l'' = l' + Z_{w'} \cdot E_{w'} + Z_{s_{w'}} \cdot E_{s_{w'}}$ and denote $h_j(x) = \sum_{i \geq 0} C(u_{w', j}, l'' + i E_{u_{w', j}}) x^i$. Then we have

\begin{equation*}
C(w', l'') = \textnormal{ Coeff} \left( (1- x)^{\delta_{w'} - 2} \cdot \prod_{1 \leq j \delta_{w'} -1} h_j(x) ,   x^{-(l', E_{w'})} \right).
\end{equation*}

Similarly as before, we get that $h_j(x) = \frac{x^{k_j}}{1-x}$ for some integers $k_j \geq 0$, therefore we have 
$C(w', l'') = 1$ if $\sum_{1 \leq j \leq \delta_{w'}-1} k_j \leq -(l'', E_{w'})$ and $C(w', l'') = 0$ if $\sum_{1 \leq j \leq \delta_{w'}-1} k_j > -(l'', E_{w'})$.

We have to prove that $\sum_{1 \leq j \leq \delta_{w'}-1} m_j \leq -(l', E_{w'})$ implies that $\sum_{1 \leq j \leq \delta_{w'}-1} k_j \leq -(l'', E_{w'})$.  
For this, we will show that $k_j \leq m_j + Z_{u_{w', j}}$, then the statement follows since $(Z, E_{w'}) = 0$ implies $\sum_{1 \leq j \leq \delta_{w'}-1} Z_{u_{w', j}} = -(l'', E_{w'}) + (l', E_{w'})$.

In order to prove $k_j \leq m_j + Z_{u_{w', j}}$, we have to show that if $l'$ is a cycle and $C(u_{w', j},  l') = 1$ then $C(u_{w', j}, l''  + Z_{u_{w', j}} \cdot E_{u_{w', j}}) = 1$. 
But we know that $C(u_{w', j}, l''  + Z_{u_{w', j}} \cdot E_{u_{w', j}})  = C(u_{w', j}, l' + Z_{u_{w', j}} \cdot E_{u_{w', j}} +  Z_{w'} \cdot E_{w'})$ and by our second induction
hypothesis $C(u_{w', j},  l') = 1$ follows that $C(u_{w', j}, l' + Z_{u_{w', j}} \cdot E_{u_{w', j}} +  Z_{w'} \cdot E_{w'}) = 1$. 
This indeed shows that $k_j \leq m_j + Z_{u_{w', j}}$, which finishes the proof of our main proposition.
\end{proof} 

\begin{remark}
Note that in the proof of the proposition the monomial conditions is used only for the branches which does not contain the vertex $v$.
\end{remark}

In the following, we show that if $l'$ is a cycle then for the number $C:= \sum_{0 \leq l \in L_{\calv}, l_v = 0} z^{\mathcal{T}}(l' + l )$  we have $0 \leq C \leq -(l', E_v) + 1$.

Let us denote $f_j(x) = \sum_{i \geq 0} C(u_{ j}, l' + i E_{u_{j}}) x^i$. Then we can write

\begin{equation*}
C = \textnormal{ Coeff} \left( (1- x)^{\delta_v - 2} \cdot \prod_{1 \leq j \leq t } f_j(x) ,   x^{-(l', E_v)} \right).
\end{equation*}

Assume first that $v$ is a node. Using part 2) of Proposition \ref{segedprop} for the vertices $u_{v, j}$, we get that $f_j(x) = \frac{x^{p_j}}{1-x}$ for some integer $p_j \geq 0$, thus we have:

\begin{equation*}
C = \textnormal{ Coeff} \left( (1- x)^{\delta_v - 2} \cdot \prod_{1 \leq j \leq t } \frac{x^{p_j}}{1-x} ,  x^{-(l', E_v)} \right).
\end{equation*}
In other words, if $p = \sum_{1 \leq j \leq t} p_j$  then 

\begin{equation*}
C = \textnormal{ Coeff} \left(  \frac{x^{p}}{(1-x)^2} ,  x^{-(l', E_v)} \right).
\end{equation*}
This means that $C = 0$ if $p > -(l', E_v)$ and $C = p + (l', E_v) + 1$ if $p \leq  -(l', E_v)$, so we are done in this case.

Assume next that $v$ is an end-vertex. Then

\begin{equation*}
C = \textnormal{ Coeff} \left( \frac{f_1(x)}{1-x} ,   x^{-(l', E_v)} \right).
\end{equation*}
This means that $C = \sum_{0 \leq i \leq -(l', E_v)} C(u_{1}, l' + i E_{u_{1}})$ and we know by Proposition \ref{segedprop} that $C(u_{ 1}, l' + i E_{u_{1}}) \in \{0, 1\}$ for all $0 \leq i \leq -(l', E_v)$. Therefore we get $0 \leq C \leq -(l', E_v) + 1$ also in this case.

Finally, if $\delta_v = 2$ then we have

\begin{equation*}
C = \textnormal{ Coeff} \left( f_1 \cdot f_2 ,   x^{-(l', E_v)} \right), 
\end{equation*}
which means that $C = \sum_{0 \leq i \leq -(l', E_v)} C(u_{1}, l' + i E_{u_{1}}) \cdot C(u_{ 2}, l' + (-(l', E_v)  - i) E_{u_{2}})$.

 By Proposition \ref{segedprop} one has that $C(u_{ 1}, l' + i E_{u_{1}}) \in \{0, 1\}$ and $C(u_{ 2}, l' + i E_{u_{1}}) \in \{0, 1\}$ for all $0 \leq i \leq -(l', E_v)$, hence the inequality $0 \leq C \leq -(l', E_v) + 1$ follows.
\vspace{0.5cm}

Recall that when the monomial conditions holds for a resolution graph, then the geometric genus of a corresponding splice quotient analytic structure coincides with the canonical normalized Seiberg-Witten invariant of the link, and so it is nonnegative. Using the results of this section, we can prove a small improvement about this nonnegativity as follows. 

\begin{corollary}
Let $\mathcal{T}$ be a resolution graph with a rational homology sphere link and assume that the monomial conditions holds for branches of nodes which
do not contain a fixed vertex $v$. Then the canonical normalized Seiberg-Witten invariant is nonnegative.
\end{corollary}

\begin{proof}

As always, $\calv$ denotes the set of vertices of $\calt$. We prove the statement by induction on the number of vertices $|\calv|$. If $|\calv| \in \{1, 2\}$ then the graph is rational and the statement follows, see eg. \ref{ss:sqat}.

Let $\mathcal{T}_1, \cdots, \mathcal{T}_t$ be the connected components of $\mathcal{T} \setminus v$ and let's denote the corresponding neighbours of the vertex $v$ by $v_1, \cdots, v_t$. Note that for the graphs $\mathcal{T}_j , 1 \leq j \leq t$,  the monomial conditions hold for branches of nodes which do not contain the vertex $v_j$, so by the induction hypothesis the  
canonical normalized Seiberg-Witten invariant of $\mathcal{T}_j$ is nonnegative for any $1 \leq j \leq t$.

On the other hand, using Theorem \ref{surgery} and Theorem \ref{dualcount} one has the following identity

\begin{equation*}
\mathfrak{sw}^{norm}_{0}(\mathcal{T}) = \sum_{ 1 \leq j \leq t} \mathfrak{sw}^{norm}_{0}(\mathcal{T}_j) + \sum_{[l'] = [Z_K], l'_v < (Z_K)_v } z^{\mathcal{T}}(l').
\end{equation*}
This implies that we need to prove the positivity $\sum_{[l'] = [Z_K], l'_v < (Z_K)_v } z^{\mathcal{T}}(l') \geq 0$.

Consider a cycle $A$ which has $E_v$-coordinate $0$ and all the other coordinates are very negative, and $Z_K -  (Z_K)_v E_v- A  \in L$. Then we have

\begin{equation*}
 \sum_{[l'] = [Z_K], l'_v < (Z_K)_v } z^{\mathcal{T}}(l') = \sum_{0 \leq k <  (Z_K)_v \atop (Z_K)_v - k \in \bZ}  \sum_{l \geq 0, l_v = 0 } z^{\mathcal{T}}(A + k E_v + l ).
\end{equation*}

Finally, by the discussion after Proposition \ref{segedprop} we know that under the conditions of this statement one has $\sum_{l \geq 0, l_v = 0 } z^{\mathcal{T}}(A + k E_v + l ) \geq 0$, 
which finishes the proof.
\end{proof}

\end{document}